\documentclass[a4paper]{amsart}

\usepackage[english]{babel}
\usepackage[latin1]{inputenc}
\usepackage{amssymb,amsmath,amsthm,amsfonts}
\usepackage{graphicx}
\usepackage[usenames,dvipsnames]{xcolor}
\usepackage{cancel}
\usepackage[colorlinks]{hyperref}
\usepackage[normalem]{ulem}
\usepackage{dsfont}
\usepackage[showonlyrefs]{mathtools}
\mathtoolsset{showonlyrefs=true}

\newcommand{\N}{\mathbb{N}}

\newcommand{\R}{\mathbb{R}}

\newcommand{\wu}{\widetilde{u}}
\newcommand{\ww}{\widetilde{w}}

\newcommand{\eps}{\varepsilon}

\newcommand{\de}[1]{\mathrm{d} #1}
\newcommand{\Om}{\Omega}
\newcommand{\mean}[1]{\,-\hskip-1.08em\int_{#1}} 

\newcommand{\cm}[1]{{\color{red}{#1}}}

\newcommand{\Efb}{\mathcal E_{q,M,\eta}}
\newcommand{\Eqme}{ E_{q,M,\eta}}
\newcommand{\q}{q}
\def\mh{\widehat{m}}

\def\N{\mathbb{N}}

\def\ut{\widetilde{u}}
\def\uh{\widehat{u}}

\def\de{\delta}

\def\hc{\mathcal{H}}
\def\H{\mathcal{H}}
\def\Omt{\widetilde{\Om}}
\def\Omh{\widehat{\Om}}

\newcommand{\resmeas}{\mathbin{\vrule height 1.6ex depth 0pt width
		0.13ex\vrule height 0.13ex depth 0pt width 1.3ex}} 

\newtheorem{proposition}{Proposition}[section]
\newtheorem{theorem}[proposition]{Theorem}
\newtheorem{corollary}[proposition]{Corollary}
\newtheorem{lemma}[proposition]{Lemma}
\theoremstyle{definition}
\newtheorem{definition}[proposition]{Definition}
\newtheorem{remark}[proposition]{Remark}
\newtheorem{conjecture}[proposition]{Conjecture}

\title{An optimal design problem for a charge qubit}

\author{Dario Mazzoleni}
\address{Dipartimento di Matematica F. Casorati\\
	Universit\`a di Pavia\\
	Via Ferrata 5, 27100 Pavia, Italy}
\email{dario.mazzoleni@unipv.it}

\author{Cyrill B. Muratov} \address{{Dipartimento di Matematica,
		Universit\`a di Pisa, Largo B. Pontecorvo, 5, 56127 Pisa, Italy}}
\email{{cyrill.muratov@unipi.it}}

\author{Berardo Ruffini}
\address{Dipartimento di Matematica\\
	Universit\`a di Bologna\\
	Piazza di Porta San Donato 5, 40126 Bologna, Italy }
\email{berardo.ruffini@unibo.it}

\thanks{{D. Mazzoleni and B. Ruffini have been partially supported
		by the MUR via PRIN project P2022R537CS.  D. Mazzoleni has been
		partially supported by the MUR via PRIN project
		P2020F3NCPX. C. B. Muratov was partially supported by NSF via
		grant DMS-1908709 and MUR via PRIN 2022 PNRR project
		P2022WJW9H. C. B. Muratov also acknowledges the MUR Excellence
		Department Project awarded to the Department of Mathematics,
		University of Pisa, CUP I57G22000700001. The authors are members
		of INdAM-GNAMPA. Finally, the authors are grateful to
		A. Bernand-Mantel for multiple enlightening discussions on the
		underlying physics of the problem and thank B.  Velichkov
		for suggesting a strategy for the proof of
		Lemma~\ref{le:lipschitz}.}}

\begin{document}

\begin{abstract}
		In this paper we introduce a simple variational model describing the
		ground state of a superconducting charge qubit. The model gives rise
		to a shape optimization problem that aims at maximizing the number
		of qubit states at a given gating voltage.  We show that for small
		values of  the charge optimal shapes exist and are
		$C^{2,\alpha}$-nearly spherical sets. 
		{In contrast, we prove that balls are not minimizers for large
			values of the charge and conjecture that optimal shapes do not exist, with the
			energy favoring disjoint collections of sets. } 
	\end{abstract}
	
	\maketitle
	
	\tableofcontents

	\section{Introduction}\label{intro}
	
	This paper studies the shape optimization problem associated with
	the energy functional
	\begin{equation}
		\label{eq:Eq}
		E_q(\Omega):=\inf_{u\in H^1_0(\Omega)} \left( \int_{\Omega}|\nabla
		u(x)|^2 \, dx +{\frac{q}{2}} \int_\Omega\int_\Omega\frac{u ^2 (x)
			\, u ^2(y)}{|x-y|}\,dx\,dy \ : \ \int_\Omega u ^2(x) dx = 1 \right),    
	\end{equation}
	for a fixed parameter $q > 0$ referred to as ``charge'' from
        now on, among all measurable sets $\Omega \subset \mathbb R^3$
        with prescribed volume. With $q = 0$, {minimizing
          $E_q{(\Omega)}$ among domains $\Omega$ with a fixed
          volume} is the classical shape optimization problem for the
        first Dirichlet eigenvalue of the Laplacian, whose solution is
        known to be a ball (see~\cite{henrot} for an overview). The
        Dirichlet energy, together with a volume constraint on
        $\Omega$, acts as a cohesive term that forces the perimeter of
        the set to be minimized. At the same time the second term in
        the definition of the energy is the Coulombic repulsive energy
        for $q > 0$, favoring separation of charges and domain
        splitting. Thus the considered problem falls into a general
        class of geometric variational problems with competing
        attractive and repulsive interactions that received a
        significant attention recently (for a broad overview, see
        \cite{cmt17}). It is the competition of these two interactions
        that makes the behavior of such problems highly non-trivial
        and interesting, and also makes their analysis challenging.
	
	The choice of the particular energy functional in
        \eqref{eq:Eq} is motivated by a basic model in quantum
        mechanics introduced by Hartree in 1927
        \cite{hartree}. We note that although this model was
        originally meant to describe the behavior of electrons in an
        atom, it suffers from a deficiency due to its lack of
        accounting for the fermionic nature of the electrons and in
        that context was superseded by a more appropriate Hartree-Fock
        model (see \cite{lions, lebris05} for a mathematical
        discussion of these models). Nevertheless, the so-called
        restricted Hartree equation
	\begin{equation}\label{eq:hartreeeq}
		-\Delta u + V u + \Big(u^2\star\frac{q}{|x|}\Big)u = \eps_q
		u \qquad \text{in } H^1(\mathbb R^3; \mathbb C), 
	\end{equation}
	where $V$ is the external potential, $\eps_q$ is the lowest energy
	level and ``$\star$'' denotes a convolution, together with its
	associated energy functional naturally re-emerge in the context of
	superconductors, in which the variable $u$ stands for the
	self-consistent ground state wave function of the Cooper pairs in the
	Bose-Einstein condensate at zero temperature. Thus, we interpret the
	energy in \eqref{eq:Eq} as a simple model for the ground state energy
	of the Cooper pairs confined to a nanoscale superconducting island
	$\Omega$ that is embedded into an insulator ($V = 0$ in $\Omega$ and
	$V = +\infty$ in $\Omega^c$) \cite{beloborodov}. The obtained energy
	is a single-orbital Hartree functional for bosons, with the total
	number of condensate particles proportional to $q > 0$. By the
	well-known property of the single-orbital Schr\"odinger operator, the
	function $u$ may be chosen to be real-valued and nonnegative
	\cite{LiebLoss}.
	
	A Cooper pair box, or a charge qubit, is an example of a quantum bit
	device that uses the charge states of the Cooper pairs in a
	superconducting nanoscale island to represent quantum information
	\cite{buttiker, bouchiat,vion,kjaergaard}. Ordinarily, a charge qubit
	is treated simply as a capacitor, with its quantum state corresponding
	to the discrete number of charges on the capacitor. Therefore, the
	shape of the island enters into the consideration solely through the
	value of the island's capacitance. At the same time, as the island's
	dimensions become smaller, as well as in the presence of a high
	dielectric constant matrix, an interplay between the kinetic and the
	potential energies of the Cooper pairs may become notable, making the
	dependence of the system's characteristics on the island shape less
	straightforward. A natural question one may then ask is whether one
	could take advantage of the superconducting island's shape to optimize
	some characteristics of the charge qubit. For example, one may ask
	what shape of the island at a given volume would maximize the number
	of stable charge states at a given gate voltage (i.e., for a given
	$V = V_0 < 0$ fixed in $\Omega$), which is equivalent to minimizing
	the energy in \eqref{eq:Eq} among all such domains. This is precisely
	the {mathematical} problem treated in the present paper.

	As it is our interest to optimize $E_q(\Omega)$ with respect to
	$\Omega$, we consider the optimal design problem
	\begin{equation}\label{eq:minmin1}
		\inf_{\substack{\Om\subset \R^3 \\ |\Om|= |B_1|}}
		E_q(\Omega)= \inf_{\substack{\Om\subset \R^3 \\ |\Om|= |B_1|}}
		\inf_{u\in H^1_0(\Omega)} \left\{ E_q(u,\Omega)
		\ : \ \int_{\Om} u^2 \, dx =1 \right\},
	\end{equation}
	where
	\begin{equation}
		\label{eq:EuOm}
		E_q(u, \Omega):= \int_{\Omega}|\nabla
		u(x)|^2 \, dx +{\frac{q}{2}} \int_\Omega\int_\Omega\frac{u ^2 (x)
                  \, u ^2(y)}{|x-y|}\,dx\,dy .   
	\end{equation}
	Notice that due to the scaling properties of the energy functional in
	\eqref{eq:EuOm} any choice of the volume of $\Omega$ may be reduced to
	that of $|\Omega| = |B_1| = \frac43 \pi$ in \eqref{eq:minmin1} by
	redefining $q$. By the same reason all the physical constants in the
	problem may be absorbed in the value of the dimensionless constant
	$q$, which is thus the only non-trivial parameter of the model (see
	Appendix \ref{appendix} for details).

        \begin{remark}[The shape optimization viewpoint]
From a purely mathematical point of view, problem
            \eqref{eq:minmin1} is a shape optimization problem in which the functional to be optimized is a nonlocal,
            nonlinear perturbation of the first eigenvalue of the
          Dirichlet Laplacian. The nonlocal nature of the
          perturbation, even under some very strong regularity
          assumptions on the state variable $u$ or the set
          $\Omega$ drastically prevents adapting standard shape
          optimization techniques. In particular, to our
            knowledge no second-order shape derivatives or
          symmetrization results are available in the literature for
          such functionals, implying that even local rigidity of
          critical points is in principle a highly nontrivial question.
        \end{remark}

	\subsection{Main {results} and detailed strategy of {their} proof}
	The main result of the paper is the following:
	\begin{theorem}\label{thm:main}
		
          For all $\eps>0$ there exists $q^*=q^*(\eps)>0$ and
          $\alpha\in (0,1)$ such that, for all $0 < q \leq q^*$,
          there exists an optimal set for
          problem~\eqref{eq:minmin1}. Furthermore, every optimal
          set $\Omega$ is $C^{2,\alpha}$-nearly spherical,
          namely, there is a function
          $\varphi_\eps\colon \partial B_1 \to \R$ of class
            $C^{2,\alpha}$ such that
          $\|\varphi_\eps\|_{C^{2,\alpha}}\leq \eps$ and
		\[ \partial\Omega=\Big\{( 1 +\varphi_\eps(x)) x : x\in \partial B_1
		\Big\}.
		\]

	\end{theorem}
	%
	
	\begin{remark}
		It is natural to
		conjecture that for sufficiently small $q$
		problem~\eqref{eq:minmin1} admits a unique minimizer: the ball.  {A
			natural strategy to prove that conjecture} would be to develop a
		quantified second order shape derivative of the energy. {To} our
		knowledge, no second order shape derivative for such a nonlocal
		energy is available at {present}, and {providing} one appears
		to be quite challenging. 
	\end{remark}

	
 {The} strategy of the proof of Theorem \ref{thm:main}  {is quite}
	involved{, hence we offer an outline here. Its overall structure}
	follows ideas developed in \cite{km,km1,gnr1,gnr2,gnr4,MurNovRuf,MNR2022} for
	isoperimetric-type variational problems and developed later on in
	\cite{maru,brdeve} for spectral optimization problems ({this} list
	is not exhaustive). In our case it can be summarized, broadly
	speaking, in three main steps.
	\begin{enumerate}
		\item Prove existence of an optimal set among equibounded sets.
		\item Prove regularity of the optimal shapes: any minimizer previously found is a nearly spherical set.
		\item Remove the equiboundedness hypothesis.
	\end{enumerate}
	
	Point (3) is almost independent and is dealt with {in} Section
	\ref{sec:unbounded}. It is in fact an improvement of spectral surgery
	techniques developed in \cite{mp,maru}, allowing to {reduce to the} uniformly bounded {setting}. The technique nevertheless
	works on connected sets, hence its proof {relies on} some mild
	regularity that we have to show in the previous sections.

	\noindent
	The core of the proof is contained in points (1) and (2). Their proof,
	performed in the (quite long) Section
	\ref{sec:existenceandregularity}, is {convoluted}. Since its
	technical details may hide the ideas behind it, we describe it here.
	\begin{itemize}
		\item First we {restrict} ourselves to consider{ing} the problem
		for equibounded sets, that is, sets uniformly contained in a large
		enough ball $B_R$, $R\gg1$. This greatly simplifies the ensuing
		compactness arguments, and does not lead to a loss of generality in
		view of point (3).
              \item Then we show the existence of a minimizer. In
                order to do that, we transform the energy $E_q$ into a
                new equivalent energy $E_{q,M}$. The related ground
                state energy is such that the $L^2$ constraint on the
                density function $u$ is replaced by a Lagrange
                multiplier\footnote{We will transform the energy two
                  more times, for a total of four {(equivalent)}
                  energies! This appears quite cumbersome, yet we are
                  not able to reduce these complications.}, see
                formula \eqref{eq:E_qM}.
		\item Next we {eliminate} the volume constraint on the admissible
		sets: following ideas from \cite{aac,brdeve,maru} we consider a new
		energy $\Omega\mapsto E_{q,M,\eta}(\Omega)$ where the volume
		constraint on $\Omega$ is replaced by a Lagrange multiplier (more
		precisely we need to use a penalizing piecewise linear function
		$\eta\mapsto f_\eta (|\Omega|)$, because of the different scalings
		of the addends in $E_{q,M}$). We are now able to show that a
		minimizer for $E_{q,M,\eta}$ exists, and it has finite
		perimeter. Such a property plays a crucial role later on in the
		proof.
		
		\vspace{0.3cm}
		\noindent
		A major complication {arises} now: at this stage we are not able to
		show that the problems of minimizing $E_{q,M}$ under volume constraint
		and the unconstrained minimization of $E_{q,M,\eta}$ are
		equivalent. To {establish this, we hence} need to show some deeper
		regularity results on the minimizers of the latter energy. {To that end:}
		
		\vspace{0.3cm}
		\item {We} consider an equivalent \emph{free boundary} formulation
		of the problem, where the minimization in $\Omega$ is replaced by
		the minimization of an energy functional $\mathcal E_{q,M,\eta}(u)$
		for $u\in H^1_0(B_R)$. Every optimal $\Omega$ happens to be the
		support of an optimal function for $\mathcal E_{q,M,\eta}$.
              \item We then adapt techniques from the free boundary
                regularity theory to show that any minimizer of
                $\mathcal E_{q,M,\eta}$ is Lipschitz continuous and
                nondegenerate, thus its positivity set is open and
                satisfies density estimates from below and
                  above. Note that here a major technical, but
                crucial, point is rendering all the implicit constants
                independent of $q$.
		\item With such a regularity at hand, by exploiting a quantitative version of the Faber--Krahn inequality we can prove that for $q$ small minimizers of $E_{M,q,\eta}$ are close to a ball {in the $L^1$ and in the Hausdorff distance}. This is just enough to show the equivalence of the (minimization of) $E_{q,M}$ and $E_{q,M,\eta}$. 
		\item Eventually, thanks to the fact that minimizers have finite perimeter, by means of a shape variation analysis {and an improvement of flatness}, we show that minimizers are in fact $C^{2,\alpha}-$regular. 
		
	\end{itemize}
	\begin{remark}
		It is quite natural to {suppose} that the closeness of minimizers
		to a ball should directly imply that the equiboundedness constraint
		might be removed. This would simplify the above general
		strategy. Nevertheless we are not able to do that as all the
		regularity estimates (in particular the crucial density
		estimates) do depend on $R$, with constants diverging as
		$R\to+\infty$.
	\end{remark}
	
	{The second result of the paper provides some information
          {about the solutions} when $q$ is large. We show that in
          {this} regime the ball {cannot} be optimal and that
          any optimal set must have diameter {at least of order
            $q^{1/2} \gg 1$}.}
	\begin{theorem}\label{thm:nomin}
{There exists a universal $\overline q>0$ such that if $q>\overline q$ then $B_1$ is not optimal for problem \eqref{eq:minmin1}. 
More precisely, any minimizer must have $\mathrm{diam}(\Omega)\geq Cq^{1/2}$, for a universal constant $C>0$.}
	\end{theorem}
        {The proof of this result follows by estimating the energy
          on a suitable union of small balls at very large mutual
          distance, as the repulsive term becomes dominant. We expect
          that this should lead to nonexistence of minimizers in the
          $q$ large regime, in view also of the result by Lu and
          Otto~\cite{LuOtto} about the {closely related}
          Thomas-Fermi-{Dirac-}Von Weizs\"acker energy. The proof of
          such a conjecture seems quite challenging, though, {due
            to the lack of a uniform bound with respect} to $q$ on the
          $L^\infty$ norm of $u$.  }
	

	\medskip {\bf Plan of the paper.}  In Section~\ref{prelim} we
        show some preliminary results on the nature of the energy
        $E_q{(\cdot, \Omega)}$ and in particular we study its
        Euler-Lagrange equation. In doing so we investigate the
        nonlinear and nonlocal eigenvalue problem associated to the
        Euler-Lagrange equation of the energy. Then we recall some
        notions about quasi-open sets and the quantitative Faber-Krahn
        inequality.  Section~\ref{sec:existenceandregularity} is
        devoted to the proof of Theorem~\ref{thm:main} {in the
          equibounded setting (namely
          Theorem~\ref{thm:mainbdd})}. More precisely, we introduce a
        new functional without {the} measure constraint, {then} prove
        existence of minimizers and {their} mild regularity properties
        (using a free boundary formulation), and finally we show the
        equivalence with problem~\eqref{eq:minmin1} {in an equibounded
          setting}.  In Section~\ref{sec:highreg}, we {first show what
          is the optimality condition at the free boundary and} then
        prove that minimizers are $C^{2,\alpha}$-nearly spherical,
        employing free boundary regularity techniques.
        Section~\ref{sec:unbounded} is devoted to a surgery argument
        which allows to conclude the proof of Theorem~\ref{thm:main}.
        Finally, in Section~\ref{sec:nonexistence} we prove
        Theorem~\ref{thm:nomin} {concerning {the} regime when
          $q$ is large}.


	\section{Preliminaries}\label{prelim}

        We present here some basic properties of the functionals that
        will be used in the proofs.  
        Throughout the paper, we adopt the following notations: for
       $\Omega \subset \R^3$ a bounded open set and
        $u\in H^1_0(\Omega)$ denoting with $\star$ the usual
        convolution, we let
	\begin{equation}\label{eq:defvuhu}
		v_u(x)={ u^2(x) \star \frac{1}{|x|} } =\int_\Omega\frac{u^2(y)}{|x-y|}\,dy,
              \end{equation}
              with $v_u \in W^{2,3}_\mathrm{loc}(\R^3)$ by
                \cite[Theorem 9.9]{GT} and Sobolev embedding.  We
              begin by establishing a uniform bound on $v_u$.
	\begin{lemma}\label{le:hardytype}
          Let $\overline x\in \R^3$, let $\Omega\subset \R^3$ be
          a bounded open set, and let
          $\varphi\in H^1_0(\Omega)$. Then
		\[
                  \int_{\Omega} \frac{\varphi^2(x)}{|x-\overline
                    x|}\,dx\leq 2
                  \|\varphi\|_{L^2(\Omega)}\,\|\nabla
                  \varphi\|_{L^2(\Omega)}.
		\]
	\end{lemma}
	\begin{proof}
          Up to extending with the value $\varphi(x)=0$ for
          $x\in \R^3\setminus \Omega$, we can suppose
          $\varphi\in H^1_0(\R^3)$. By H\"older inequality we
          obtain
		\[
		\int \frac{\varphi^2(x)}{|x-\overline x|}\,dx\le \left \| \frac{\varphi(\cdot)}{|\cdot-\overline x|} \right\|_{L^2(\R^3)}\|\varphi\|_{L^2(\R^3)} 
		\]
		Then we can apply the classical Hardy-Sobolev
                inequality, see for instance~\cite[Corollary 2 of
                Section~2.1.7]{mazya}
		\[
		\left\|\frac{\varphi}{|\cdot-\overline x|}\right\|_{L^2(\R^3)}\leq
		2 \|\nabla \varphi\|_{L^2(\R^3)},
		\]
		to obtain the result.
	\end{proof}
	
	We recall now that for $\phi,\psi : \Omega \to \R$
        the Coulomb {energy}  denoted by 
	\begin{equation}\label{eq:defDcoulomb}
          \qquad
          D(\phi,\psi)=\int_\Omega\int_\Omega\frac{\phi(x)\psi(y)}{|x-y|}\,dx\,dy 
        \end{equation}
        is a well defined positive definite bilinear form, provided
        $D(|\phi|, |\psi|) < \infty$ \cite[Theorem 9.8]{LiebLoss}. We
        note that in the definition of $D(\cdot,\cdot)$ the dependence
        on $\Omega$ is implicit, as it is the domain of the functions
        $\phi,\psi$.

\subsection{Minimization of $E_q(u, \Omega)$ in $u$.}        	
Existence of $u$ achieving the infimum in the definition of
  $E_q(\Omega)$ is an application of the Direct Method in
the Calculus of Variations. 
	\begin{lemma}\label{le:prelim1}
          Let $\Omega\subset \R^3$ be a bounded open set and
          let $q>0$.  Then the minimization problem 
		\begin{equation}\label{eq:EuOmega}
                  E_q(\Omega)= \inf\Big\{E_q(v,\Omega) : v\in H^1_0(\Omega),\;
                  \int_\Omega v^2\, dx=1\Big\}. 
		\end{equation}
		admits a solution with constant sign (nonnegative,
                  without loss of generality). 
	\end{lemma}
	\begin{proof}
		Let $(u_n)_n$  be a minimizing sequence for the energy. As all the addends in the definition of $E_q(u,\Omega)$ are positive, we infer that $(u_n)_n$ is bounded in $H^1_0(\Omega)$.

		Then up to passing to a subsequence, $(u_n)_n$ converges weakly in
		$H^1_0(\Omega)$ and strongly in $L^2(\Omega)$ to some function
		$u\in H^1_0(\Omega)$. In particular the $L^2-$convergence implies
		that $\|u\|_{L^2(\Omega)}=1$. By lower semicontinuity with respect to
		the weak convergence, we have that
		\[
		\int |\nabla u|^2\,dx\le \liminf_{n\to+\infty} \int |\nabla u_n|^2\,dx
		\]
		and, by Fatou Lemma 
		\[
		D(u^2,u^2)=\iint_{\Omega\times\Omega}\frac{u^2(x)u^2(y)}{|x-y|}\,dxdy\le \liminf_{n\to+\infty} D(u_n^2,u_n^2).
		\]
		
		Hence, the energy is lower semicontinuous
		\[
		E_q(u,\Omega)\le  \liminf_{n\to+\infty} E_q(u_n,\Omega),
		\]
		and $u$ is a minimizer.
		The fact that $u$ can be chosen of constant sign follows since \[
		E_q(|u|,\Omega)= E_q(u,\Omega),
		\]
		and since $|u|$ is  an admissible competitor. 
              \end{proof}
	
	\begin{remark}\label{rmk:scaleinvariant}
          It is not difficult to check that the
          scale invariant functional
		\[
		\widetilde E_q(v,\Omega):=\frac{\int_{\Omega}|\nabla v|^2\, dx}{\|v\|_{L^2}^2}
		+\frac{q}{2}\frac{\int_\Omega\int_\Omega\frac{v^2(x)v^2(y)}{|x-y|}\,dx\,dy}{\|v\|_{L^2}^4},\qquad v\in H^1_0(\Omega),
		\]
		leads to an equivalent, yet unconstrained, minimization problem
		\[
		\min\Big\{E_q(v,\Omega) : v\in H^1_0(\Omega),\; \int_\Omega v^2\,
		dx=1\Big\}=\min\Big\{\widetilde E_q(v,\Omega) : v\in H^1_0(\Omega)
		{, v \not\equiv 0}\Big\}.
		\]
		In particular we can always choose an optimal function $u\in H^1_0(\Omega)$ for the unconstrained problem which satisfies, a posteriori, the constraint $\|u\|_{L^2}=1$.
	\end{remark}
	
Let us introduce the
        semilinear operator on $H^1_0(\Omega)$,
	\begin{equation}\label{eq:elloperator}
		L_{\Omega,q} {(u)} =-\Delta
		u+q\Big(u^2\star\frac{1}{|\cdot|}\Big)u . 
	\end{equation}
	One easily shows that $L_{\Omega,q}$ is a positive operator, that is,
	$(L_{\Omega,q} {(u)},u)_{L^2}\ge0$.  We say that a number
	$\lambda>0$ is a (nonlinear) eigenvalue for $L_{\Omega,q}$ if
	there exists a non-null function $v$ such that
	\begin{equation}\label{eq:ellpde}
		L_{\Omega,q}  {(v)}=\lambda v, \qquad {\int_\Omega v^2 \, dx = 1,}
	\end{equation}
	in distributional sense.
	Such a function is then called an eigenfunction corresponding to $\lambda$. 
	
	By means of a first variation, it is possible to check that
        (as it is observed in~\cite{lions}) the optimal function {$u$}
        attaining $E_q(\Omega)$ is a (nonnegative) normalized
        eigenfunction for the operator $L_{\Omega,q}$ associated to
        the eigenvalue
	\[ \lambda_q=\int_\Omega|\nabla u|^2 {\, dx} +qD(u^2,u^2).
	\]
	We note that for $q\in (0,1]$ and $|\Omega| = |B_1|$ the
          quantity $\lambda_q$ is uniformly bounded from below and above:
		\begin{equation}\label{eq:unifboundlambdaq}
			\lambda_0(B_1)\leq
                        \lambda_0(\Omega)\leq \lambda_q\leq 2
                          E_q(\Omega) \leq 2 E_1(B_1),
		\end{equation}
		where $\lambda_0(A)$ denotes the first eigenvalue of
                the Dirichlet Laplacian of an open set $A\subset\R^3$,
                and we used the Faber-Krahn inequality in the first
                inequality of~\eqref{eq:unifboundlambdaq}.  {Here
                  we are {limiting} our analysis to the sets
                  satisfying $E_q(\Omega)\leq E_q(B_1)$.  This is not
                  restrictive since otherwise the ball would be
                  optimal for $E_q$.}  In other words, the
                Euler-Lagrange equation for $E_q(\Omega)$ is
	\begin{equation}\label{eq:ELhartree}
		L_{\Omega,q}  {(u)} =\lambda_q u,\qquad u\in H^1_0(\Omega).
	\end{equation}

	By classical elliptic theory we deduce now a useful uniform bound on
	the $L^\infty$ norm of solutions to~\eqref{eq:ELhartree}. Precisely, we
	have the following result. 
	\begin{lemma}\label{lem:inftybound}
          Let $\Omega \subset \R^3$ be a bounded open set, let
            $q\in(0,1]$ and let $u$ be a solution of
		\begin{equation}\label{eq:LOmegaqeq}
			L_{\Omega,q} {(u)} =\lambda_q u \qquad\text{ in }
			{H^1_0(\Omega)},\qquad{\int_\Omega u^2\,dx=1}.
		\end{equation}
		Then $u\in C^\infty(\Omega)$ and, hence, is a
                classical solution of \eqref{eq:ELhartree}.
                  Furthermore, if $\lambda_q \leq M$ then
			\[
			\|u\|_{L^\infty(\Omega)}\le C,
			\]
			for a constant $C=C(M ,|\Omega|)>0$
                          depending only on $M$ and $|\Omega|$. 
	\end{lemma}

	\begin{proof}
           Thanks to Lemma~\ref{le:hardytype}, and by the Young
            inequality, we deduce that
            \[ \|v_u\|_{L^\infty(\Omega)}\leq 
                  1 +\|\nabla u\|^2_{L^2(\Omega)} \leq
                1 +\lambda_q .
			\]					
                        Hence the function
\begin{equation}\label{eq:rhs}
		c(x)=\lambda_q-q v_u(x)
\end{equation}
		is equibounded in $\Omega$ by a constant depending
                only on {$\lambda_q$}.  
                As $u$ solves 
		\[
		\Delta u +c(x)u=0,\qquad\text{{in $H^1_0(\Omega)$,}}
		\]
		by \cite[Theorem 8.8]{GT} we can assert that
                $u\in H^2_{\rm loc}(\Omega)$ and by a simple bootstrap
                argument, $u\in C^\infty(\Omega)$ so that it is a
                classical solution of \eqref{eq:ELhartree}. 
            Eventually, by \cite[Theorem 8.15]{GT} we conclude that
		\[
                  \|u\|_{L^\infty(\Omega)}\le
                  C(M,|\Omega|)\|u\|_{L^2(\Omega)}=C(M,|\Omega|), 
		\]
		where $C(M,|\Omega|) > 0$ is a constant depending
                only on $M$ and $|\Omega|$. 
	\end{proof}

	\begin{corollary}\label{cor:usuperharm}
    There exists $q_0\in (0,1]$ such that for every
            $q \in (0, q_0]$ and every bounded open set
            $\Omega \subset \R^3$ with $|\Omega| = |B_1|$ every
          nonnegative minimizer $u$ of $E_q(u, \Omega)$
          solves
		\[
                  -\Delta u=c(x) u \quad \text{in} \
                    H^1_0(\Omega),\qquad \int_\Omega u^2\, dx=1,
		\]
		for some $c\in C^\infty(\Omega; \R^+)$. In
                  particular, $u$ is superharmonic.
	\end{corollary}
	
	\begin{proof}
          We set $c(x)=\lambda_q-q v_u(x)$.  The positivity of $c$
          for all $q$ sufficiently small universal follows by the
          uniform $L^\infty$ bound proven in Lemma
          \ref{lem:inftybound} and by the uniform bounds on
            $\lambda_q$ from~\eqref{eq:unifboundlambdaq}. Since
          $u\geq 0$, then it is superharmonic for $q$ below a certain
          threshold $q_0$.
	\end{proof}

	\noindent From now on $q_0>0$ always refers to the
          universal constant in Corollary~\ref{cor:usuperharm}.

          \begin{remark}
            We note that in general one cannot expect that a minimizer
            $u$ of $E_q(u, \Omega)$ has the entire set $\Omega$ as its
            support. This can can be already seen in the case $q = 0$
            with $\Omega$ consisting of two disjoint sets whose first
            Dirichlet eigenvalues are distinct. It is not difficult to
            show that this situation persists to the case of $q > 0$
            sufficiently small depending on $\Omega$.
        \end{remark}

      On the other hand, we can prove that the energy $E_q$ is the
      same on $\Omega$ and on the positivity set of $u$ and give some
      further characterizations.
	\begin{lemma}\label{lem:cc2}
          Let $q > 0$, let $\Omega$ be a bounded open set and let
          $u\in H^1_0(\Om)$ be a nonnegative minimizer of
            $E_q(u, \Omega)$.  Then
          $E_q(\Omega)=E_q(\{u>0\})$. Furthermore, if $\Omega_i$
            is a connected component of $\Omega$ then either $u>0$ or
            $u\equiv0$ on $\Omega_i$.

	\end{lemma}
	\begin{proof}
          It was already observed that we can choose $u\geq 0$, {moreover, as $u$ is continuous, $\{u>0\}$ is an open set}. By the
          optimality of $u$, $E_q(\Omega)\geq E_q(\{u>0\})$. On the
          other hand, $\{u>0\}\subset \Omega$ and by monotonicity of
          the functional $E_q{(\Omega)}$ {in $\Omega$}, we infer
          $E_q(\Omega)=E_q(\{u>0\})$. The last part of the
            statement follows from Lemma \ref{lem:inftybound} and
            \cite[Theorem~9.10]{LiebLoss}.
	\end{proof}

	\begin{remark}\label{rmk:connectedcomp}
          We highlight again that if $\Omega$ is not connected, then
         we only know up to this point that
          the set $\{u>0\}$ coincides with the union of some of
          the connected components of $\Omega$ and has the same energy
          as $\Omega$. 
	\end{remark}

        We show now that $E_q(u, \Omega)$ has a unique (up to
          sign) minimizer $u$ {if $\Omega$ is a \emph{connected}
            set (or, more generally, if $u>0$ on all
            connected components of} $\Omega$). To do so, we follow
        an approach proposed in~\cite{BenguriaBrezisLieb} (later
        revisited in~\cite{brafra2012}), based on a \emph{hidden
          convexity property} of the functional.

	\begin{proposition}\label{prop:hiddenconvexity}
          Let $q > 0$ and let $\Omega$ be a connected
            bounded open set.  Let $u,v\in H^1_0(\Omega)$ be such
          that $\|u\|_{L^2(\Omega)}=\|v\|_{L^2(\Omega)}=1$. Let
          $\sigma_t=\sigma_t(u,v)$ be defined as
		\[
		\sigma_t=((1-t)u^2 + t v^2)^{1/2},\qquad \text{for }t\in(0,1).
		\]
		Then the map $g:t\mapsto E_q(\sigma_t,\Omega)$ is
                strictly convex. In particular, the function $u$
                  attaining the minimum of $E_q(u, \Omega)$ is
                unique.
	\end{proposition}
	
	\begin{proof}
          {Thanks to Lemma~\ref{lem:cc2} and the connectedness of
            $\Omega$, we deduce that the optimal function $u$ for
            $E_q(\Omega)$ is strictly positive in $\Omega$. Then the
            claim follows directly
            from~\cite[Lemma~4]{BenguriaBrezisLieb}, which assures
            that the map $t\mapsto E_q(\sigma,\Omega)$ is strictly
            convex.  Eventually, if $u,v$ are two distinct (strictly
            positive) minimizers, since
            $\|\sigma_t(u,v)\|_{L^2(\Omega)}=1$ by strict convexity it
            follows immediately that
            $E_q(\sigma_t,\Omega)<E_q(u,\Omega)=E_q(v,\Omega)$, a
            contradiction.  }
	\end{proof}

	As a side result of Proposition~\ref{prop:hiddenconvexity}, we have also the radiality of $u$, when $\Omega$ is a ball.
	\begin{corollary}
          Let $\Omega\subset \R^3$ be a ball. Then the unique
          solution to
          \[ \min\Big\{E_q(u,\Omega) : u\in H^1_0(\Omega),\;
            \int_\Omega u^2\, dx=1\Big\},
		\] 
		is radial.
	\end{corollary}
	\begin{proof}
		The uniqueness of the solution follows from Proposition~\ref{prop:hiddenconvexity}.
		Concerning the radiality, {it follows from the rotational invariance of the energy $E_q(u,\Omega)$.}
	\end{proof}

	\subsection{Fraenkel asymmetry and quantitative Faber-Krahn inequality}
	Here we recall the sharp quantitative version
	of the Faber--Krahn inequality. We first remind the notion of
	\emph{Fraenkel asymmetry}: for a set $\Omega\subset\R^3$ with finite
	measure we define 
	\begin{equation}\label{eq:fraenel}
		\mathcal A(\Omega)=\inf_{x\in\R^3}\frac{|\Om\Delta (B+x)|}{|\Omega|},
	\end{equation}
	where $B$ denotes the ball of measure $|\Omega|$ centered at
        the origin.  We also recall that we denote by
        $\lambda_0(A)$ the first eigenvalue of the Dirichlet
        Laplacian of a set $A\subset \R^3$.
	
	\begin{theorem}{\cite{brdeve}}\label{thm:quantitativefk}
		There exists a universal positive constant $\widehat \sigma>0$ such
		that for all open sets $\Om\subset\R^3$ with finite measure we
		have
		\[ |\Om|^{2/3}\lambda_0(\Om)-|B_1|^{2/3}\lambda_0(B_1)\geq \widehat \sigma
		\mathcal A(\Om)^2,
		\]
		where $\mathcal A$ the Fraenkel asymmetry.
	\end{theorem}
	
	{
		\subsection{Some facts about quasi-open sets}\label{ssec:quasiopen}
		Finally, we recall some definitions and facts about quasi-open sets which will be needed in the following.
		\begin{definition}\label{def:quasiopen}
			A \emph{quasi-open} set is a measurable set $\Omega\subset \R^3$ such that for all
			$\varepsilon>0$ there exists $K_\varepsilon$ compact such that its Newtonian capacity
			${\rm cap}(K_\varepsilon)<\varepsilon$ and
			$\Omega\setminus K_\varepsilon$ is open.  
			Similarly, a function $u:\Omega\to\R$ is
			\emph{quasi-continuous} if for all $\varepsilon>0$ there exists a
			compact set $K_\varepsilon$ such that
			${\rm cap}(K_\varepsilon)<\varepsilon$ and the restriction of $u$ to
			$\Omega\setminus K_\varepsilon$ is continuous. Eventually, we say
			that a property holds \emph{quasi-everywhere} on a set if it holds
			up to a set of null-capacity.
		\end{definition}  
		It is well-known that every $u\in H^1(B_R)$ admits a
		quasi-continuous representative $\widetilde u$. Moreover, if
		$\widetilde u$ and $\widehat u$ are two quasi-continuous
		representatives of $u$, then they are equal
		quasi-everywhere. Therefore, in this paper, for every
		$u\in H^1(B_R)$ we identify it with its quasi-continuous
		representative.  A quasi-open set is then simply {a superlevel set}
		of (the quasi-continuous representative of) a function
		$u\in H^1(B_R)$.  For more details on quasi-open sets and
		quasi-continuous functions, and the definition of capacity we refer
		to~\cite[Chapter 2]{henrot},~\cite[Chapter~3]{HenrotPierre} or~\cite[Chapter~2 and~4]{Heinonen}.
		
		Let us also stress that it is standard to define the Sobolev space $H^1_0$ on a quasi-open set $\Omega\subset \R^3$ as\[
		H^1_0(\Omega)=\{u\in H^1(\R^3) : u=0\text{ quasi-everywhere in }\R^3\setminus \Omega\}.
		\]
		If $\Omega$ is an open set, this definition coincides
                with the usual one,
                see~\cite[Section~3.3.5]{HenrotPierre}  for more
                details. 
		\begin{remark}\label{rmk:qoLinfty}
                  If $\Omega\subset \R^3$ is a quasi-open set and
                  $u\in H^1_0(\Omega)$ is a nonnegative solution
                  to~\eqref{eq:LOmegaqeq}, then
                  $u\in L^\infty(\Omega)$ and its $L^\infty$ norm is
                  bounded by a constant as in
                  Lemma~\ref{lem:inftybound}. This can be checked with
                  an approximation argument directly from the
                  definition of a quasi-open set.
		\end{remark}
	}
	
	\section{An existence result for an auxiliary problem}\label{sec:existenceandregularity}
	In this section we begin the  proof of the following result, which can be seen as an equibounded version of Theorem~\ref{thm:main}.
	\begin{theorem}\label{thm:mainbdd}
		There exists $R_0>1$ such that, for all $R\geq R_0$ there exists $\overline q=\overline q(R)>0$ such that for all $\q\leq \overline q$ there exists a minimizer for the problem
		\begin{equation}\label{eq:mainpb}
			\min\Big\{E_q(\Omega) : \Omega\subset B_R,\;|\Omega|=|B_1|\Big\}.
		\end{equation}
		Moreover, for all $\varepsilon>0$ there exists $q_\eps=q_\eps(R)>0$ such that if $q\le q_\eps$ then any minimizer is an $(\eps,C^{2,\alpha})-$nearly spherical set, with $\alpha=\alpha(\eps)\in(0,1)$.
	\end{theorem}
	The proof of this result is {carried} out through several steps, divided in the next (sub)sections, as outlined in the introduction.
	
	\subsection{Removing the $L^2$ norm constraint in $E_q$}
	\begin{proposition}\label{prop:vincoloL2}
 Let $\Omega\subset \R^3$ be a
          bounded quasi-open set of measure $|B_1|$ and let
          $q\leq q_0$.  There exists $\overline M>0$ {(uniform for all
            $q\leq q_0$)}
          such that for all $M\geq \overline M$, the minimizers of the
          problem~\eqref{eq:EuOmega} are the same as those of
		\begin{equation}\label{eq:minnoL^2}
			\min\Big\{E_{q,M}(v,\Omega) : v\in H^1_0(\Omega)\Big\},
		\end{equation}
		where 
		\begin{equation}\label{eq:EqM}
			E_{q,M}(v,\Omega)=E_q(v,\Omega)+M\Big|\int v^2\, dx-1\Big|.
		\end{equation}
		
	\end{proposition}
	\begin{proof}
		Problem~\eqref{eq:minnoL^2} admits a minimizer, as it can be seen as in the proof of Lemma~\ref{le:prelim1}.
		We set \[
		0<\alpha=\min{\Big\{E_q(v,\Omega) : v\in H^1_0(\Omega),\; \int v^2\, dx=1\Big\}}=E_q(u,\Omega).
		\]
		Since $\int u^2=1$, clearly for all $M>0$, $  E_{q,M}(u,\Omega)=E_q(u,\Omega)$, hence \[
		\alpha\geq \min{\Big\{  E_{q,M}(v,\Omega) : v\in H^1_0(\Omega)\Big\}}.
		\]
		Let us assume by contradiction that there is a sequence $M_k\to+\infty$ such that \[
		\min{\Big\{  E_{q,M_k}(v,\Omega) : v\in H^1_0(\Omega)\Big\}}= E_{q,M_k}(\widehat u_k,\Omega)<\alpha.
		\]
		We call {$\sigma_k = \|\widehat u_k\|_{L^2}-1$} and define
		$u_k=\frac{\widehat u_k}{1+\sigma_k}$, so that
		$\|u_k\|_{L^2}=1$. {This definition makes sense for all large
			enough $k$, since} $\sigma_k\to 0$ as $k\to+\infty$.  We can then compute\[
		\begin{split}
			\alpha> E_{q,M_k}(\widehat u_k,\Omega)&=(1+\sigma_k)^2\int|\nabla u_k|^2\, dx+\frac{\q}{2}(1+\sigma_k)^4D(u_k^2,u_k^2)+M_k\Big|(1+\sigma_k)^2-1\Big|\\
			&=E_{q}(u_k,\Omega)+(2\sigma_k+o(\sigma_k))\int|\nabla u_k|^2\, dx+\frac{\q}{2}(4\sigma_k+o(\sigma_k))D(u_k^2,u_k^2) +2M_k(|\sigma_k|+o(\sigma_k)).
		\end{split}
		\]
		Noting that $\int|\nabla u_k|^2\, dx$ and
                $D(u_k^2,u_k^2)$ are uniformly bounded by $2\alpha$
                and $q\leq q_0$, for
                $M_k\geq 10\alpha$ and $k$ large enough, we conclude
                that
                \[ \alpha> E_{q,M_k}(\widehat u_k,\Omega)\geq
                  E_q(u_k,\Omega)\geq \alpha,
		\] 
		a contradiction. We note that the choice of $M_k$ for
                reaching a contradiction is independent of $q$ and
                $\Omega$.
	\end{proof}
	
	In accordance with the notation of the previous theorem, we set 
	\begin{equation}\label{eq:E_qM}
		E_{q,M}(\Omega):=\min\Big\{E_q(v,\Omega)+M\Big|\int v^2\, dx-1\Big| : v\in H^1_0(\Omega)\Big\}
	\end{equation}
	and note that, as a consequence of Proposition~\ref{prop:vincoloL2}, if $u_\Omega$ is a minimizer of $E_{q,M}(\Omega)$ for $M\geq \overline M$, then $\int u_\Omega^2\, dx=1$, thus, a posteriori, $E_q(\Omega)=E_{q,M}(\Omega)$.

		\begin{remark}\label{rmk:M}
                  From now on, we fix once and for all a constant
                  $M>\overline M$, and we stress that this constant
                  does not depend on $q \leq q_0$ and will not be
                  changed later in the paper.
		\end{remark}

	\subsection{Another auxiliary problem: {removing} the volume
		constraint}
	In order to get rid of the measure constraint, we follow an approach first proposed by Aguilera, Alt and Caffarelli~\cite{aac}.
	Let $\eta\in(0,1)$ and consider the piecewise linear function \[
	f_\eta\colon \R^+\rightarrow \R,\qquad f_\eta(s)=
	\begin{cases}
		\eta(s-|B_1|),\qquad \text{if }s\leq |B_1|,\\
		\frac{1}{\eta}(s-|B_1|),\qquad \text{if }s\geq |B_1|.
	\end{cases}
	\]
	It is easy to check that, for all $0\leq s_2\leq s_1,$ there holds
	\begin{equation}\label{eq:propfeta}
		\eta(s_1-s_2)\leq f_\eta(s_1)-f_\eta(s_2)\leq \frac{1}{\eta}(s_1-s_2).
	\end{equation}

	We introduce then the functional (recalling that $M$ is fixed, see Remark~\ref{rmk:M} and Proposition~\ref{prop:vincoloL2}),
	\begin{equation}
		\label{eq:Eqmeta}
		E_{q,M,\eta}(\Om):=E_{q,M}(\Om)+f_\eta(|\Om|),
	\end{equation}
	and, for $R>1$, the  minimization problem 
	\begin{equation}\label{eq:minGeta2}
		\min\left\{E_{q,M,\eta}(\Om) : \Om\subset B_R,\text{ $\Om$ open}\right\}.
	\end{equation}
	To prove the existence of a minimizer for
        problem~\eqref{eq:minGeta2}, we need first to work in the
        setting of quasi-open sets {(see
          Section~\ref{ssec:quasiopen})} and then recover the
        regularity.

	Thus, we first focus on the problem:
	\begin{equation}\label{eq:minGeta1}
		\min\left\{E_{q,M,\eta}(\Om) : \Om\subset B_R,\text{ $\Om$ quasi-open}\right\}.
	\end{equation}
	
	We aim to prove that problem~\eqref{eq:minGeta1} is equivalent to the
	{constrained} problem~\eqref{eq:mainpb}, at least for $\q$ and $\eta$ small
	enough. To do that we first have to prove existence and some mild
	regularity of minimizers of $E_{q,M,\eta}$.  We begin by showing a
	lower bound for $E_{q,M,\eta}$ on equibounded sets.
	\begin{lemma}\label{le:boundbelowR}
          Let $R>1$, $q\in (0,q_0]$ and $\eta\in (0,1)$. Then,
          for all quasi-open $\Om\subset B_R$, we have
		\[ E_{q,M,\eta}(\Om)\geq
		\lambda_0(B_1)R^{-2}-{|B_1|}. 
		\]
	\end{lemma}
	\begin{proof}
		Let $v\in H^1_0(\Omega)$, with $\int_\Omega v^2\, dx=1$ be a function attaining the infimum in the definition of $E_{q,M}(\Omega)=E_q(\Omega)$, so that  
		\[
		E_{q,M,\eta}(\Omega)=\int|\nabla v|^2\, dx+\frac{\q}{2}D(v^2,v^2)+f_\eta(|\Omega|).
		\]
		By the monotonicity of Dirichlet eigenvalues, the inclusion
		$\Omega\subset B_R$, the positivity of $D(\cdot, \cdot)$, the scaling
		properties of $\lambda_0$, and the
		definition of $f_\eta$, we obtain 
		\[
		\int|\nabla v|^2\, dx+\frac{\q}{2}D(v^2,v^2)+f_\eta(|\Omega|)\geq
		\lambda_0(\Omega){-\eta\max\{|B_1|-|\Omega|,0\}}\ge \lambda_0(B_1)R^{-2}-|B_1|,
		\]
		so the claim is proved.
	\end{proof}
	
	The following existence result is mostly an adaptation to our situation  of~\cite[Lemma~4.6]{brdeve} and~\cite[Lemma~3.2]{maru}, which are in turn inspired by \cite[Theorem~2.2 and Lemma~2.3]{bu}. 
	
	\begin{lemma}\label{le:existminG}
          Let $\eta\in(0,1)$, $\q\in (0,q_0]$ and let $R>1$.
          There exists a {minimizer} for problem~\eqref{eq:minGeta1}.
          Moreover all minimizers have perimeter uniformly bounded by
          a constant depending on $R,\eta$.
	\end{lemma}
	\begin{proof}
		Let $(\Om_n)_n\subset B_R$ be {a sequence of smooth sets} such that
		\[
		E_{q,M,\eta}(\Om_n)\leq \inf\left\{E_{q,M,\eta}(\Om) : \Om\subset B_R,\; \text{quasi-open}\right\}+\frac{1}{n}.
		\] 
		Let $u_n$ be an optimal function with unit $L^2$ norm
                for the minimization of $E_q(\Om_n)$ or equivalently
                $E_{q,M}(\Omega_n)$, so that by Lemma~\ref{lem:cc2},
                either $\Om_n=\{u_n>0\}$, or $\{u_n>0\}$ is a union of
                some connected components\footnote{Notice that
                  {since $\Omega_n$ are smooth then} $u_n$ are
                  regular functions and there are no
                  measure-theoretical issues in defining connected
                  components.} of $\Omega_n$, with
                $E_{q,M,\eta}(\{u_n>0\})\leq E_{q,M,\eta}(\Omega_n)$,
                so that $(\{ u_n>0\})_n$ still forms a minimizing
                sequence and we can replace $\Omega_n$ with
                $\{u_n>0\}$. Let $t_n=1/\sqrt{n}$. We define
		\[ \widetilde \Om_n:=\{u_n>t_n\}.
		\]  
		We have \[
		E_{q,M,\eta}(\Om_n)\leq E_{q,M,\eta}(\widetilde \Om_n)+\frac{1}{n},
		\]
	 and since we can take $\widetilde u_n:=(u_n-t_n)_+$ as
		a competitor in the minimization problem defining
		$E_M(\widetilde \Omega_n)$, we obtain
		\begin{equation}\label{eq:ecomp}
			\begin{split}
				&\int_{\{u_n>0\}}|\nabla u_n|^2\, dx+\frac{\q}{2}D(u_n^2,u_n^2)+M\Big|\int u_n^2\, dx-1\Big|+f_\eta(|\{u_n>0\}|)\\
				&\leq \int_{\{u_n>t_n\}}|\nabla u_n|^2\, dx+\frac{\q}{2}D(\widetilde u_n^2,\widetilde u_n^2)+M\Big|\int(\widetilde u_n)^2\, dx-1\Big|+f_\eta(|\{u_n>t_n\}|)+\frac{1}{n}.
			\end{split}
		\end{equation}
		Using the uniform bound on the $L^\infty$ norm of $u_n$, see Lemma~\ref{lem:inftybound}, we observe that \[
		\begin{split}
			\int_{\{u_n>t_n\}}u_n^2-(u_n-t_n)^2\, dx&=\int_{\{u_n>t_n\}}2t_nu_n-t_n^2\, dx\leq 2\|u_n\|_{L^\infty}t_n|\{u_n>t_n\}|\leq C(R)t_n|\{u_n>t_n\}|,\\
			\text{ and }\quad \int_{\{0\leq u_n\leq t_n\}}u_n^2\, dx&\leq t_n^2|\{0< u_n\leq t_n\}|, 
		\end{split}
		\] 
		and obtain the estimate (possibly increasing the value of $C(R)$)
		\begin{equation}\label{eq:star}
			\int_{\{u_n>0\}}u_n^2\, dx-\int_{\{u_n>t_n\}}(u_n-t_n)^2\, dx\leq C(R)t_n|\{u_n>0\}|.
		\end{equation}
		Noting that $D(\widetilde u_n^2,\widetilde u_n^2)-D(u_n^2,u_n^2)\leq 0$, recalling the property~\eqref{eq:propfeta} of $f_\eta$ and~\eqref{eq:star}, we can rewrite~\eqref{eq:ecomp} as 
		\begin{equation}\label{eq:e1}
			\int_{\{0<u_n<t_n\}}|\nabla u_n|^2\, dx+\eta|\{0<u_n<t_n\}|\leq C(R)M t_n|\{u_n>0\}|+\frac1n\leq C(R,M)t_n+\frac1n.
		\end{equation}
		On the other hand, since $\eta<1$, using coarea formula, the arithmetic-geometric-mean inequality and~\eqref{eq:e1}, we obtain\[
		\begin{split}
			&2\eta\int_{0}^{t_n}P(\{u_n>s\})\,ds=2\eta\int_{\{0<u_n<t_n\}}|\nabla u_n|\,dx\\
			&\leq \eta\int_{\{0<u_n<t_n\}}|\nabla u_n|^2\,dx+\eta|\{0<u_n<t_n\}|\leq C(R,M)t_n+\frac1n.
		\end{split}
		\]
		As $t_n=1/\sqrt{n}$ we can find a level $0<s_n<1/\sqrt{n}$ such that
		the sets $W_n:=\{u_n>s_n\}$
		satisfy\[ P(W_n)\leq \frac{2}{
			t_n}\int_0^{t_n}P(\{u_n>s\})\,ds\leq
		\frac{C(R,M)}{\eta}+\frac{1}{\eta t_n n}\leq
		C(R,M,\eta)+\frac{1}{\eta\sqrt{n}}.
		\]
		It is easy to check that $(W_n)_n$ is still a  minimizing sequence for problem~\eqref{eq:minGeta1}. In fact, with arguments similar to the ones used above, we obtain:
		\begin{equation}\label{eq:e2}
			\begin{split}
				E_{q,M,\eta}(W_n)&
				=\int_{\{u_n>s_n\}}|\nabla u_n|^2\, dx+\frac{\q}{2}D\Big((u_n-s_n)_+^2,(u_n-s_n)_+^2)\Big)\\
				&\hspace{30pt}+M\Big|\int_{\{u_n>s_n\}}(u_n-s_n)^2\, dx-1\Big|+f_\eta(|\{u_n>s_n\}|)\\
				&\leq E_{q,M,\eta}(\Om_n)+C(R,M)s_n+f_\eta(|\{u_n>s_n\}|)-f_\eta(|\{u_n>0\}|)\\
				&\leq E_{q,M,\eta}(\Om_n)+\frac{C(R,M)}{\sqrt{n}}-\eta|\{0<u_n<s_n\}|\leq E_{q,M,\eta}(\Om_n)+\frac{C(R,M)}{\sqrt{n}},
			\end{split}
		\end{equation}
		where we used  that $D\Big((u_n-s_n)_+^2,(u_n-s_n)_+^2\Big)-D(u_n^2,u_n^2)\leq 0$ and property~\eqref{eq:star} with $s_n$ in place of $t_n$. 
		Moreover, since the sets of the sequence $(W_n)_{n }$ have equibounded
		perimeter, there exists a Borel set $W_\infty$ such that (up to
		passing to subsequences)
		\begin{equation}\label{eq:WnL1}
			W_n\rightarrow W_\infty,\text{ in }L^1(B_R),\qquad P(W_\infty)\leq
			C(R,M,\eta).
		\end{equation}
		On the other hand, an optimal function (normalized in $L^2$) $w_n$ attaining $E_{q,M,\eta}(W_n)$, is equibounded in $H^1(B_R)$. In fact, being $(W_n)_n$ a minimizing sequence for $E_{q,M,\eta}$, we have that   
		\[
		\int |\nabla w_n|^2\, dx+\int w_n^2\, dx\le 1+ E_{q,M,\eta}(W_n)\leq 1+C.
		\]
		Hence, up to passing to subsequences, there is $w\in H^1_0(B_R)$ such
		that
		\begin{equation}\label{eq:wnae}
			w_n\rightarrow w \quad \text{strongly in }L^2(B_R),\text{
				weakly in }H^1_0(B_R)\text{ {and pointwise a.e.}}
		\end{equation}
		Let $W:=\{w>0\}$, and recall that we are identifying $w$ with its
		quasi-continuous representative. Then, {thanks to~\eqref{eq:WnL1} and~\eqref{eq:wnae}, we deduce} 
		\[
		\chi_W(x)\leq \liminf_{n\to +\infty}\chi_{W_n}(x)=\chi_{W_\infty}(x),\qquad \text{ for a.e. }x\in B_R,
		\]
		hence $|W\setminus W_\infty|=0$, that is $W\subset W_\infty$ up to a negligible set.
		We now observe that $\Omega\mapsto f_\eta(|\Omega|)$ is continuous with respect to the $L^1$ convergence of sets, the Dirichlet energy is lower semicontinuous with respect to the weak $H^1$ convergence and the  functional $D(\cdot,\cdot)$ is lower semicontinuous with respect to the strong $L^2$ convergence by Fatou lemma. We can therefore pass to the limit in~\eqref{eq:e2} and obtain 
		\begin{equation}\label{eq:signfeta}
			\begin{split}
				&\int_W |\nabla w|^2\, dx+\frac{\q}{2}D(w^2,w^2)+f_\eta(|W_\infty|)\leq \liminf_{n\to+\infty}\int_{W_n}|\nabla w_n|^2\, dx+\frac{\q}{2}D(w_n^2,w_n^2)+f_\eta(|W_\infty|)\\
				&\leq \liminf_{n\to+\infty}E_{q,M,\eta}(W_n)=\inf_{\Om\subset B_R}E_{q,M,\eta}(\Om)\leq E_q(W)+f_\eta(|W|).
			\end{split}
		\end{equation} 
		In conclusion, using also~\eqref{eq:signfeta}, we have \[
		\eta|W_\infty\setminus W|=\eta(|W_\infty|-|W|)\leq f_\eta(|W_\infty|)-f_\eta(|W|)\leq 0,
		\]
		thus $|W_\infty\setminus W|=0$, which entails $W=W_\infty$ a.e. and this is the desired minimizer for problem~\eqref{eq:minGeta1}.
	\end{proof}

	\subsection{Free boundary formulation}\label{sec:fb}
	In the previous section we introduced the functional $E_{q,M,\eta}$,
	see \eqref{eq:Eqmeta}, we proved existence and mild regularity
	properties of minimizers (namely: they are sets of finite
	perimeter). In this section we improve the regularity for such
	sets. This will be needed in the sequel but allows us also to show the
	equivalence between unconstrained minimizers of $E_{q,M,\eta}$ and
	volume constrained minimizers of $E_q$ and $E_{q,M}$. The crucial
	remark is that one may consider, in place of the shape functional
	$E_{q,M,\eta}$, a functional defined on the larger space $H^1_0(B_R)$
	and take a \emph{free boundary} approach. Let us define, for
	$u\in H^1_0(B_R)$, 
	\begin{equation}\label{eq:Eqmefb}
		\Efb(u)=\int_{\{u>0\}}|\nabla u|^2\, dx+\frac{\q}{2}\int_{\{u>0\}}\int_{\{u>0\}}\frac{u^2(x)u^2(y)}{|x-y|}\,dxdy+M\left|\int_{\{u>0\}} u^2\, dx-1\right|+f_\eta(|\{u>0\}|),
	\end{equation}
	{and we note that one could have equivalently integrated over $B_R$ in all the integrals above.}
	\begin{lemma}\label{lem:equiv}
          Let $\eta\in (0,1)$, $q\in (0,q_0]$ and $\Omega$ be a
          minimizer for problem~\eqref{eq:minGeta1}.  Then
          {every} minimal function of $E_{q}( \Omega)$ is a minimizer
          of $\Efb$. Viceversa, if $w$ minimizes $\Efb$, then
          $\Omega=\{w>0\}$ is a minimizer of $\Eqme$.  Furthermore, it
          is possible to select a minimizer $\Omega$ for
          $E_{q,M,\eta}$ which coincides with the support of {an}
          optimal function for $E_q(\Omega)$.
	\end{lemma}
	\begin{proof}
		Concerning the first claim, let $\Omega$ be an optimal set for $E_{q,M,\eta}$ and $u$ an optimal function for $E_q(\Omega)$. We immediately note that $\{u>0\}=\Omega$ up to sets of zero measure, otherwise $\{u>0\}\subset\Omega$, thus $f_\eta(|\{u>0\}|)<f_\eta (|\Omega|)$ so that $\Eqme(\Omega)>\Eqme(\{ u>0\})$, a contradiction with the optimality of $\Omega$. As a consequence, $\Efb(u)=E_{q,M,\eta}(\{u>0\})=E_{q,M,\eta}(\Omega)$.
		Therefore, for all $v\in H^1_0(B_R)$, we have \[
		\Efb(u)=E_{q,M,\eta}(\Omega)\leq E_{q,M,\eta}(\{v>0\})=\Efb(v),
		\]
		namely $u$ is a minimizer for $\Efb$.
		
		Let us focus on the second claim: let $u$ be an optimal function for $\Efb$, and we call $\Omega=\{u>0\}$. For all $\widetilde\Omega\subset B_R$, calling $\widetilde u$ any optimal function attaining $E_q(\widetilde \Omega)$, we have
		\[
		E_{q,M,\eta}(\Omega)=\Efb(u)\le \Efb(\widetilde u)\le E_{q,M,\eta}(\widetilde\Omega),
		\]
		as requested. Notice that the last inequality is not a priori an equality, as in principle $\{\widetilde u>0\}\subset \widetilde  \Omega$.
		
		Concerning the last part of the statement, it is
                enough to notice that, given an optimal function
                  $u$ for $E_q(\Omega)$, $\{u>0\}$ is always a
                minimizer for $E_{q,M,\eta}$.
	\end{proof}

	In the {rest} of this section, we focus on the new (equivalent)
	formulation of problem~\eqref{eq:minGeta1}
	\begin{equation}\label{eq:pbminfb}
		\min\left\{ \Efb(u)\,:\,u\in H^1_0(B_R)  \right\}.
	\end{equation} 
	As a consequence of Lemma~\ref{lem:equiv} and
	Remark~\ref{rmk:connectedcomp}, working on problem~\eqref{eq:pbminfb}
	means selecting an optimal set $\{u>0\}$ for the original
	problem~\eqref{eq:minmin1} or for~\eqref{eq:minGeta1}, namely the
	union of the connected components of $\Omega$ where $u$ is nonzero.
	
	{
		In the next results, we  work with a minimizer of the form $\Omega=\{u>0\}$. This is not restrictive, as next lemma shows.
	}
	
	%
	\begin{lemma}\label{le:supportok}
		{For all $R\geq 1$, $\eta\in (0,1)$ and $ q\leq q_0$, if $\Omega$ is an optimal set for problem~\eqref{eq:minGeta1} and $u$ is an associated (nonnegative) optimal function attaining $E_q(\Omega)$, then $\{u>0\}=\Omega$ (up to a negligible set).
			Moreover, $\Omega$ is connected.
		}
	\end{lemma}
	\begin{proof} 
          { Let us argue for the sake of contradiction and assume that
            $\{u>0\}$ is strictly contained in $\Omega$.  We already
            know from Remark~\ref{rmk:connectedcomp} that the above
            assumptions entail that $u\equiv0$ on a connected
            component $\omega$, with positive measure. Therefore,
            $f_\eta(|\{u>0\}|)<f_\eta(|\Omega|)$ and we have
            contradicted the minimality of $\Omega$ for
            $E_{q,M,\eta}$.  } { Concerning the last part of the
            statement, if $\Omega=\{u>0\}$ is the disjoint union of
            two components $\Omega_1$ and $\Omega_2$, by increasing
            the distance between the two components we are
            strictly decreasing the Coulomb energy term, while
            the other terms of $E_{q,M,\eta}$ are unchanged. Thus we
            constradict again the optimality of $\Omega$.}
	\end{proof}
	
	\begin{remark}
		We stress that if we could prove that an optimal function for
		\eqref{eq:pbminfb} is a quasi-minimizer for the functional
		\[
		J(u)=\int_{\{u>0\}}|\nabla u|^2\, dx+|\{u>0\}|,
		\]
		i.e. letting
		$J_{x,r}(u)=\int_{\{u>0\}\cap B_r(x)}|\nabla u|^2+|\{u>0\}\cap
		B_r(x)|$, we have, {for some $\beta>0$,} 
		\[
		J_{x,r}(u)\leq (1+K r^\beta)J_{x,r}(v),\qquad \text{for all admissible $v$ and for all $x,r$},
		\]
		this would strongly simplify the regularity proof, see for
		example~\cite{velectures}. Unfortunately, this does not seem to be the
		case in our setting, due to the presence of the nonlocal double
		integral term $D(u^2,u^2)$. Therefore we use a careful modification of
		the standard free boundary regularity techniques developed starting
		from~\cite{alca}.
	\end{remark}

	\begin{lemma}\label{le:lemma4.9}
          Let $R>1$, $\eta\in (0,1)$, $\q\in (0,q_0]$, let $\Om$ be an
          optimal set for problem~\eqref{eq:minGeta1},
          and let $u\in H^1_0(\Omega)$ be any
          (nonnegative) function attaining
          $E_q(\Omega)=E_{q,M}(\Omega)$.
          Then for every $\kappa\in (0,1)$ there are positive
          constants $K_0,\rho_0$ depending only on $\kappa, \eta, R$
          such that the following assertion holds: if
          $\rho\leq \rho_0$ and $x_0\in \overline{B_R}$, then
		\begin{equation}\label{eq:nondegmean}
			\mean{\partial B_\rho(x_0)\cap B_R}{u\,d\mathcal H^{2}}\leq K_0 \rho\,\,\,\,
			\Longrightarrow \,\,\,\,u\equiv0 \text{ in }\,\,B_{\kappa\rho}(x_0)\cap B_R.
		\end{equation}
	\end{lemma}
	\begin{proof}
		By Lemma~\ref{le:supportok}, we know that $\Omega=\{u>0\}$, that $u$ is optimal for problem~\eqref{eq:pbminfb} and has unitary $L^2$ norm.
We extend $u$ to zero outside $B_R$, so that by  Lemma \ref{lem:inftybound},  $u$ solves distributionally $-\Delta u\leq \gamma_1$ in $\R^3$ where (since $q\leq q_0$) 
		\[
		\gamma_1:=2\sup_x\left(\lambda_q-\q v_u(x) \right)\|u\|_{L^\infty(\Omega)}>0. 
		\]
		The positivity of $\gamma_1$ follows by the bound (uniform in $q$) on $\|u\|_{L^\infty(\Omega)}$ {(see Lemma~\ref{lem:inftybound})} and, consequently, on $\|v_u\|_{L^\infty(\Omega)}$, while $\lambda_q\ge \lambda_0(\Omega)\ge\lambda_0(B_1)>0$.
		Then the function \[
		x\mapsto u(x)+\gamma_1\frac{|x-{x_0}|^2-\rho^2}{6}
		\]
		is subharmonic in $B_\rho$ (recalling that we are in three dimensional setting).
		Thus, for every $\kappa\in(0,1)$, there exists $c=c(\kappa)$ such that 
		\begin{equation}\label{eq:4.20}
			\delta_\rho:=\sup_{B_{\sqrt{\kappa}\rho}{(x_0)}}u\leq c\left(\mean{\partial B_\rho{(x_0)}\cap B_R}{u\,d\mathcal H^{2}}+\gamma_1 \rho^2\right)\leq c(K_0\rho+\rho^2).
		\end{equation}
		Let us show now that there exists $\rho>0$ small enough so that there exists a positive solution $w$ of \begin{equation}\label{www}
			\begin{cases}
				-\Delta w=\frac{M}{2}(u+w),\qquad &\text{in }B_{\sqrt{\kappa}\rho}{(x_0)}\setminus B_{\kappa \rho}{(x_0)},\\
				w=\delta_\rho,\qquad &\text{on }\partial B_{\sqrt{\kappa}\rho}{(x_0)},\\
				w=0,\qquad &\text{on }B_{\kappa\rho}{(x_0)},
			\end{cases}
		\end{equation}
		where $M$ is fixed large enough so that the statement of Proposition~\ref{prop:vincoloL2} holds. 
		To show existence of a solution, let $\rho>0$ be such that 
		\[
		\alpha(\rho):=\lambda_0(B_{\sqrt{\kappa}{\rho}}\setminus B_{\kappa\rho})^{-1}\left(\frac M2 \|u\|_{L^2(B_{\sqrt{\kappa}{\rho}}{(x_0)}\setminus B_{\kappa\rho}{(x_0)})}+\frac M2 \right)\le\frac14.
		\]
		This can be easily obtained as $\lambda_0(B_{\sqrt{\kappa}{\rho}}\setminus B_{\kappa\rho})\to+\infty$ as $\rho\to0$ and since $u\in L^\infty(\Omega)$. 
		Then any minimizing sequence  for the energy 
		\begin{equation}\label{eq:energy}
			\varphi\mapsto \frac12\int_{B_{\sqrt{\kappa}{\rho}}{(x_0)}\setminus B_{\kappa\rho}{(x_0)}}|\nabla \varphi|^2\, dx-\frac M2 \int_{B_{\sqrt{\kappa}{\rho}}{(x_0)}\setminus B_{\kappa\rho}{(x_0)}}\left({\frac{\varphi}{2}}+u\right)\varphi\, dx
		\end{equation}
		with boundary conditions as in \eqref{www}
		can be chosen, after standard computations, made of nonnegative functions (as passing to the modulus decreases the energy) and such that 
		\[
		  \int_{B_{\sqrt{\kappa}\overline{\rho}}{(x_0)}\setminus B_{\kappa\rho}{(x_0)}}|\varphi_n|^2\, dx\le \alpha(\rho)\int_{B_{\sqrt{\kappa}{\rho}}{(x_0)}\setminus B_{\kappa\rho}{(x_0)}}|\nabla \varphi_n|^2\, dx\le \frac14 \int_{B_{\sqrt{\kappa}{\rho}}{(x_0)}\setminus B_{\kappa\rho}{(x_0)}}|\nabla \varphi_n|^2\, dx.
		\]
		Hence the sequence $(\|\varphi_n\|_{H^1})$ is uniformly bounded and a positive minimizer for the energy~\eqref{eq:energy} exists and solves its   Euler-Lagrange equation, namely \eqref{www}. 
		By standard elliptic regularity, we obtain that the $L^\infty$ norm of any positive solution $w$ is bounded by a constant $\gamma_2>0$ depending only on $\kappa$ and $M$. 
		By definition, $w\geq u$ on $\partial B_{\sqrt{\kappa}\rho}{(x_0)}$, therefore the function \[
		v=
		\begin{cases}
			u,\qquad &\text{in }\R^3\setminus B_{\sqrt{\kappa}\rho}{(x_0)},\\
			\min\{u,w\},\qquad &\text{in }B_{\sqrt{\kappa}\rho}{(x_0)},
		\end{cases}
		\]
		satisfies \[
		v\leq u,\qquad \{v>0\}\subset \{u>0\},\qquad \{v>0\}\setminus B_{\sqrt{\kappa}\rho}\cm{(x_0)}=\{u>0\}\setminus B_{\sqrt{\kappa}\rho}{(x_0)},
		\]
		thus $D(v^2,v^2)\leq D(u^2,u^2)$ (and we can neglect these contributions in the following computations).
		Since $v\in H^1_0(B_R)$, inequality~\eqref{eq:pbminfb} gives\[
		\begin{split}
			&\int_{B_{\sqrt{\kappa}\rho}{(x_0)}} |\nabla u|^2\, dx+f_\eta(|\{u>0\}|)\\
			&\leq \int_{B_{\sqrt{\kappa}\rho}{(x_0)}} |\nabla v|^2\, dx+M\Big|\int_{B_{\sqrt{\kappa}\rho}{(x_0)}}(v^2-u^2)\, dx\Big|+f_\eta(|\{v>0\}|).
		\end{split}
		\]
		As $v=0$ in $B_{\kappa\rho}$,  using also~\eqref{eq:propfeta}, we get 
		\[
		\begin{split}
			\eta|\{u>0\}\cap B_{\kappa\rho}{(x_0)}|&\leq \eta|(\{u>0\}\setminus \{v>0\})\cap B_{\sqrt{\kappa}\rho}{(x_0)}|\\
			&\leq f_\eta(|\{u>0\}|)-f_\eta(|\{v>0\}|).
		\end{split}
		\]
		On the other hand, we can rewrite the term involving the $L^2$ norm of the functions as\[
		\Big|\int_{B_{\sqrt{\kappa}\rho}{(x_0)}}(v^2-u^2)\, dx\Big|=\int_{B_{\sqrt{k}\rho}{(x_0)}}(u^2-v^2)\, dx=\int_{B_{\kappa \rho}{(x_0)}}u^2\, dx+\int_{(B_{\sqrt{\kappa}\rho}{(x_0)}\setminus B_{\kappa \rho}{(x_0)})\cap \{u>w\}}(u^2-w^2)\, dx.
		\]
		Thanks to the two inequalities above and the definition of $v$, we can infer 
		\begin{equation}\label{eq:4.21}
			\begin{split}
				&\int_{B_{\kappa\rho}{(x_0)}}|\nabla u|^2\, dx+\eta|\{u>0\}\cap B_{\kappa\rho}|\leq \int_{B_{\kappa\rho}{(x_0)}}|\nabla u|^2\, dx+f_\eta(|\{u>0\}|)-f_\eta(|\{v>0\}|)\\
				&\leq \int_{B_{\sqrt{\kappa}\rho}{(x_0)}\setminus B_{\kappa\rho}{(x_0)}}(|\nabla v|^2-|\nabla u|^2)\, dx+M\int_{B_{\kappa \rho}{(x_0)}}u^2\, dx+M\int_{(B_{\sqrt{\kappa}\rho{(x_0)}}\setminus B_{\kappa \rho}{(x_0)})\cap \{u>w\}}(u^2-w^2)\, dx\\
				&\leq 2\int_{(B_{\sqrt{\kappa}\rho{(x_0)}}\setminus B_{\kappa\rho}{(x_0)})\cap\{u>w\}}(|\nabla w|^2-\nabla u\cdot\nabla w)\, dx+M\int_{B_{\kappa \rho}{(x_0)}}u^2+M\int_{(B_{\sqrt{\kappa}\rho}{(x_0)}\setminus B_{\kappa \rho}{(x_0)})\cap \{u>w\}}(u^2-w^2)\, dx.
			\end{split}
		\end{equation} 
		On the other hand testing \eqref{www} with $(u-w)_+$ and integrating over ${B_{\sqrt{\kappa}\rho}\setminus B_{\kappa\rho}}$, we obtain 
		\begin{equation}\label{eq:4.22}
			\int_{(B_{\sqrt{\kappa}\rho}{(x_0)}\setminus B_{\kappa\rho}{(x_0)})\cap \{u>w\}}(|\nabla w|^2-\nabla u\cdot\nabla w)\, dx+\frac{M}{2}\int_{(B_{\sqrt{\kappa}\rho}{(x_0)}\setminus B_{\kappa \rho}{(x_0)})\cap \{u>w\}}(u^2-w^2)\, dx=\int_{\partial B_{\kappa\rho}{(x_0)}}\frac{\partial w}{\partial \nu}u\,d\mathcal H^{2},
		\end{equation}
		where $\nu$ denotes the outer unit normal exiting from $B_{\kappa\rho}$ and thanks to the fact that $w=0$ on $\partial B_{\kappa\rho}{(x_0)}$ and $w\geq u$ on $\partial B_{\sqrt{\kappa}\rho}{(x_0)}$.
		Recalling that $\|u\|_{L^\infty}\leq \gamma_1$ and $\|w\|_{L^\infty}\leq\gamma_2$, we now fix $\gamma_3=\frac{M}{2}(\gamma_1+\gamma_2)$ and consider the solution to the problem \[
		\begin{cases}
			-\Delta \widetilde w=\gamma_3,\qquad &\text{in }B_{\sqrt{\kappa}\rho}{(x_0)}\setminus B_{\kappa \rho}{(x_0)},\\
			\widetilde w=\delta_\rho,\qquad &\text{on }\partial B_{\sqrt{\kappa}\rho}{(x_0)},\\
			\widetilde w=0,\qquad &\text{on }B_{\kappa\rho}{(x_0)},
		\end{cases}
		\]
		since the torsion function on an annulus is explicit (see~\cite{brdeve}), with a direct computation one obtains\[
		\left|\frac{\partial \widetilde w}{\partial \nu}\right|\leq \beta_1\frac{\delta_\rho+\rho^2}{\rho},\qquad \text{on }\partial B_{\kappa\rho}{(x_0)},
		\]
		for some $\beta_1=\beta_1(\kappa,M)$.
		By comparison, since $\widetilde w- w$ is superharmonic (by~\eqref{www}) it follows
		\[\left|\frac{\partial w}{\partial \nu}\right|\leq \left|\frac{\partial \widetilde w}{\partial \nu}\right|\leq \beta_1 \frac{\delta_\rho+\rho^2}{\rho},\qquad \text{on}\qquad \partial B_{\kappa\rho}{(x_0)}.\] 
		We can now combine~\eqref{eq:4.21} and~\eqref{eq:4.22} to obtain
		\begin{equation}\label{eq:4.23}
			\int_{B_{\kappa\rho}{(x_0)}}|\nabla u|^2\, dx+\eta|\{u>0\}\cap B_{\kappa\rho}{(x_0)}|\leq \beta_1(\kappa)\frac{\delta_\rho+\rho^2}{\rho}\int_{\partial B_{\kappa\rho}{(x_0)}}u\,d\mathcal H^{2}+M\int_{B_{\kappa\rho}{(x_0)}}u^2\, dx.
		\end{equation}
		
		By   the uniform bound on the $L^\infty$ norm of $u$, we have the estimate\[
		M\int_{B_{\kappa\rho}{(x_0)}}u^2\, dx\leq \beta_2\delta_\rho^2|\{u>0\}\cap B_{\kappa \rho}{(x_0)}|,
		\]
		for some $\beta_2(\kappa, M)$.

		Then, using the definition of $\delta_\rho$, the trace inequality in $W^{1,1}$ and the arithmetic geometric mean inequality we obtain \[
		\begin{split}
			&\int_{\partial B_{\kappa\rho}{(x_0)}}u\,d\mathcal H^{2}\leq C(\kappa)\left(\frac{1}{\rho}\int_{B_{\kappa\rho}{(x_0)}}u\, dx+\int_{B_{\kappa\rho}{(x_0)}}|\nabla u|\, dx\right)\\
			&\leq \beta_3\left(\left(\frac{\delta_\rho}{\rho}+\frac12\right)|\{u>0\}\cap B_{\kappa\rho}{(x_0)}|+\frac12\int_{B_{\kappa\rho}{(x_0)}}|\nabla u|^2\, dx\right),
		\end{split}
		\]
		for some $\beta_3=\beta_3(\kappa)>0$.
		By collecting the above estimates, recalling again~\eqref{eq:4.20} we have, for all $\rho\leq \rho_0$ \[
		\begin{split}
			&\eta\int_{B_{\kappa\rho}{(x_0)}}|\nabla u|^2\, dx+\eta|\{u>0\}\cap B_{\kappa\rho}{(x_0)}|\\
			&\leq \beta_1\frac{\delta_\rho+\rho^2}{\rho}\int_{\partial B_{\kappa\rho}{(x_0)}}u\,d\mathcal H^{2}+\delta_\rho(1+\delta_\rho\beta_2)|\{u>0\}\cap B_{\kappa\rho}{(x_0)}|\\
			&\leq \beta_1(c(K_0+\rho)+\rho)\int_{\partial B_{\kappa\rho}{(x_0)}}u\,d\mathcal H^{2}+c(K_0\rho+\rho^2)(1+c(K_0\rho+\rho^2)\beta_3)|\{u>0\}\cap B_{\kappa\rho}{(x_0)}|\\
			&\leq \beta_1\beta_3(c(K_0+\rho)+\rho)\left[ \left(\frac{\delta_\rho}{\rho}+\frac12\right)|\{u>0\}\cap B_{\kappa\rho}{(x_0)}|+\frac12\int_{B_{\kappa\rho}{(x_0)}}|\nabla u|^2\, dx\right]\\
			&\hspace{50pt}+c(K_0\rho+\rho^2)(1+c(K_0\rho+\rho^2)\beta_2)|\{u>0\}\cap B_{\kappa\rho}{(x_0)}|\\
			&\leq \beta_1\beta_3(c(K_0+\rho)+\rho)\left(2c(K_0+\rho)+\frac12\right)\left[ \int_{B_{\kappa\rho}{(x_0)}}|\nabla u|^2+|\{u>0\}\cap B_{\kappa\rho}{(x_0)}|\right].
		\end{split}
		\]
		Eventually, by choosing $K_0,\rho_0<\overline \rho$ small enough so that \[
		\beta_1\beta_3(c(K_0+\rho_0)+\rho_0)\left(2c(K_0+\rho_0)+\frac12\right)\leq \eta/4,
		\] 
		we conclude that $u\equiv 0$ in $B_{\kappa\rho}$, for all $\rho\leq\rho_0$.
	\end{proof}
	\begin{remark}\label{rmk:nondeg}
{The statement of Lemma~\ref{le:lemma4.9} and in particular~\eqref{eq:nondegmean} can be also equivalently stated as 
\[
			\|u\|_{L^\infty(B_\rho(x_0))}\leq K_0 \rho\,\,\,\,
			\Longrightarrow \,\,\,\,u\equiv0 \text{ in }\,\,B_{\kappa\rho}(x_0)\cap B_R,
\]
see for example~\cite[Remark~4.3]{maru}.
		}
{In other words,} Lemma~\ref{le:lemma4.9} implies that if $x_0\in \overline \Omega$,
		then there is a constant {$C=C(R,\eta)>0$},
		which can be taken independent of $x_0$, such
		that   \[\sup_{B_\rho(x_0)\cap B_R}u\geq C\rho,\qquad {\text{and}\qquad \mean{\partial B_\rho(x_0)\cap B_R}{u\,d\mathcal H^{2}}\geq C\rho}.\]
	\end{remark}

	\begin{lemma}\label{le:lipschitz}
		Let $R$, $\eta$, $\q$, $\Om$ and $u$ be as in
		Lemma~\ref{le:lemma4.9}. The function $u$ can be extended to a
		Lipschitz continuous function defined in the whole $B_R$, with
		Lipschitz constant $L=L(R,\eta)$. In particular,
		$\Omega=\{u>0\}\subset B_R$ is an open set.
	\end{lemma}
	\begin{proof}
		We follow the approach of~\cite[Section~3.2]{velectures}, first proposed in~\cite{BrianconHayouniPierre}.
		
		\noindent{\it Step 1.} We prove an estimate on the nonnegative Radon measure $|\Delta u|$, namely
		\begin{equation}\label{eq:laplest}
                  |\Delta u|(B_r{(x_0)})\leq C
                  r^2,\qquad{\text{for all $x_0\in B_R$ and
                      {$0 < r < 1$} such that $B_{2r}(x_0)\subset B_R$}}
		\end{equation}
		for a universal constant $C>0$. Let $\psi\in C^\infty_c({B_{2r}(x_0)})$ for some ${B_{2r}(x_0)}\subset B_R$ , with $\|\psi\|_{L^\infty}\leq c$, and we test the optimality of $u$ against $u+\psi$, obtaining:\[
		\int_{\{u>0\}} |\nabla u|^2\, dx+\frac{\q}{2}D(u^2,u^2)+f_\eta(|\{u>0\}|)\leq \int_{\{{u+\psi>0}\}} |\nabla (u+\psi)|^2\, dx+\frac{\q}{2}D\Big((u+\psi)^2,(u+\psi)^2\Big)+f_\eta(|\{u+\psi>0\}|),
		\]
		which implies 
		\[
		-{2} \int_{{B_{2r}(x_0)}}\nabla u\cdot \nabla \psi\, dx\leq
		\int_{{B_{2r}(x_0)}}|\nabla \psi|^2\, dx+C_\eta|\{u=0\}\cap
		{B_{2r}(x_0)}|+\frac{\q}{2}\int_{{B_{2r}(x_0)}}\int_{{B_{R}(0)}}P(x,y)\,dxdy
		\]
		where 
		\[
                  P(x,y)=\frac{4u(x)\psi(x)u^2(y)+4u(x)\psi(x)\psi^2(y)+2
                    {\psi^2(x)
                      u^2(y)}+4u(x)\psi(x)u(y)\psi(y)+\psi^2(x)\psi^2(y)}{|x-y|}.
		\]
		Recalling that $\|\psi\|_{L^\infty}\leq c$ {and}
                using Lemma~\ref{lem:inftybound}, {we can
                  control the nonlocal term
                  {as}\[
                    \int_{{B_{2r}(x_0)}}\int_{{B_{R}(0)}}P(x,y)\,dxdy\leq
                    C_1
                    \int_{{B_{2r}(x_0)}}\int_{{B_{R}(0)}}\frac{1}{|x-y|}\,dxdy\leq
                    C_2 {R^2} |B_{2r}(x_0)|\leq C_3r^3.
\] }
		{Thus }we obtain 
		\begin{equation}\label{eq:estgrad}
			\begin{split}
                          &- {2} \int_{{B_{2r}(x_0)}}\nabla u\cdot
                          \nabla \psi\, dx\leq
                          \int_{{B_{2r}(x_0)}}|\nabla \psi|^2\,
                          dx+C_\eta|\{u=0\}\cap
                          {B_{2r}(x_0)}|+{Cqr^3}.
			\end{split}
		\end{equation}
		We now set, for all $\varphi\in
                C^\infty_c({B_{2r}(x_0)})$,
                $\psi={\pm} r^{3/2}\|\nabla
                \varphi\|_{L^2}^{-1}\varphi$ {and
                  from~\eqref{eq:estgrad} we deduce, for some
                  $\widetilde C>0$}
		\[
		\Big|\int_{{B_{2r}(x_0)}}\nabla u\cdot \nabla \varphi\, dx\Big|\leq {\widetilde C} r^{3/2}\|\nabla \varphi\|_{L^2({B_{2r}(x_0)})}
		\]
		It is then enough to choose
                $\varphi\in C^\infty_c(B_{2r}{(x_0)})$ with
                $\varphi\geq 0$ and $\varphi=1$ in ${B_{r}(x_0)}$
                and with
                $\|\nabla \varphi\|_{L^\infty(B_{2r})}\leq \frac2r$
                (notice that this is compatible with the requirement
                $\|\psi\|_{L^\infty}\leq c$ {indepdendently of
                  $r$}) to obtain, for some constant $C>0$:
		\begin{equation}\label{eq:estlapl}
                  |\Delta u|({B_r(x_0)})\leq|\Delta u|(\varphi)=
                  \left|\int_{{B_{2r}(x_0)}}\nabla u\cdot\nabla
                    \varphi\,dx\right| \leq C r^2. 
		\end{equation}
		
		\noindent{\it Step 2.} We prove that the Laplacian estimate~\eqref{eq:estlapl} of Step 1
		entails (recall that $\mathcal H^2(\partial B_r)=4\pi {r^2}$)
		\begin{equation}\label{eq:intest}
                  \frac{1}{4\pi {r^2}}\int_{\partial
                    {B_r(x_0})}u\,d\mathcal H^2\leq {u(x_0)}+ Cr\qquad
                  {\text{for all }x_0\in B_R}, 
		\end{equation}
		for some constant $C>0$. {This follows from~\cite[Lemma~3.6]{BrianconHayouniPierre}, which assures that, for all $x_0\in B_R$, it holds
			\begin{equation}\label{eq:BHP3.6}
				\frac{1}{4\pi r^2}\int_{\partial B_r(x_0)}u\,d\mathcal H^2-u(x_0)=\int_0^r\frac{1}{4\pi s^2}\Delta u(B_s(x_0))\,ds.
			\end{equation} 
			It is then enough to put together~\eqref{eq:BHP3.6} and~\eqref{eq:estlapl} to obtain~\eqref{eq:intest}.
		}
{Now, let us take $x_0\in \partial \{u>0\}\cap B_R$ and a sequence of $x_n\to x_0$ such that $u(x_n)=0$ for all $n$ and with $x_n\in B_{r_1}(x_0)\subset B_R$. For those points~\eqref{eq:intest} reads as
\begin{equation}\label{eq:xn}
  \frac{1}{4\pi {r^2}}\int_{\partial
                    {B_r(x_n})}u\,d\mathcal H^2\leq u(x_n)+Cr=Cr,\qquad \text{for all }r< r_1,
\end{equation}
and the constant $C$ does not depend on $n$. 
Since $u\in H^1(B_R)$, the map $x\mapsto \tfrac{1}{4\pi {r^2}}\int_{\partial
                    {B_r(x})}u\,d\mathcal H^2$ is continuous, see~\cite[Remark~3.6]{velectures}.
We can then pass to the limit as $n\to\infty$ in~\eqref{eq:xn} to deduce \[
\frac{1}{4\pi {r^2}}\int_{\partial
                    {B_r(x_0})}u\,d\mathcal H^2\leq Cr,\qquad \text{for all }r< r_1.
\]
}			
{Finally, passing to the limit  as $r\to 0$, we obtain that $u(x_0)=0$ (recalling that we are considering the quasi continuous representative of the Sobolev function $u$), thus
                  $\Omega\cap \partial\Omega=\{u>0\}\cap
                  \partial\{u>0\}=\emptyset$, hence $\Omega=\{u>0\}$
                  is an open set. }
		
		\noindent{\it Step 3.} We conclude, using Step 2, that $u$ is
		Lipschitz continuous in $B_R$, as in~\cite[Lemma~3.5]{velectures}, see also~\cite[Theorem~3.1 and~4.1]{BrianconHayouniPierre}.
	\end{proof}
	
	An immediate and fundamental consequence of
        Lemmas~\ref{le:lemma4.9} and~\ref{le:lipschitz} is the
        following density estimate on $\Omega$.
	\begin{lemma}\label{le:lemma4.11}
          Let $R > 1$ and $\eta \in (0, 1)$. There exists
            $q_1\in (0, q_0]$ such that for all
            $q\in (0,q_1 ]$, calling $\Om$ an optimal set for
            problem~\eqref{eq:minGeta2} and $u$ a positive normalized
            function attaining $E_{q,M,\eta}(\Omega)$, there exist
          {positive} constants $\theta= \theta(R,\eta)$ and
          $\rho_0=\rho_0(R,\eta)<1$ such that for every
          $x_0\in \partial \Omega$ and every $\rho\leq \rho_0$, we
          have
		\begin{equation}\label{eq:4.30}
			\theta\leq \frac{|\Omega\cap B_\rho(x_0)|}{|B_\rho|}{\leq (1-\theta)}.
		\end{equation}
	\end{lemma}
	\begin{proof}
		{Let us start from the lower bound.}
		We can assume that $x_0=0\in \partial \Omega=\partial\{u>0\}$. Thus, the nondegeneracy
		condition of Remark~\ref{rmk:nondeg} implies that \[
		\|u\|_{L^\infty(B_{\rho/2})}\geq C\tfrac{\rho}{2}.
		\] 
		Thus, there is a point $y\in B_{\rho/2}$ such that $u(y)\geq C\tfrac{\rho}{2}$.
		On the other hand, the Lipschitz continuity of $u$, with constant $L=L(R,\eta)$, implies that $u>0$ on a ball with radius $\tfrac{\rho}{2}\min\{1,\tfrac{C}{L}\}$, and so we conclude. 

		{ The upper bound can be obtained as in \cite{alca},
                  see also~\cite[Section~5.1]{velectures}, with a few
                  modifications.  Precisely, let
                  $x_0=0\in \partial\Omega$ and consider the function
                  $h$ which is equal to $u$ outside $B_\rho$, and
                    inside $B_\rho$ it is the solution of
			\[
			\begin{cases}
                          -\Delta h =\gamma_1 \quad&\text{in }B_\rho,\\
                          h = u\quad&\text{on~} \partial B_\rho,
			\end{cases}
			\]
			where
                        $\gamma_1=\gamma_1(\rho):=2\sup_x\left(\lambda_q-\q
                          v_u(x) \right)\|u\|_{L^\infty(B_\rho)}>0$.
                        As a consequence, we obtain that
                        $-\Delta (h-u)= \gamma_1-(\lambda_q u-q
                        v_u)u\geq 0$ in $B_\rho$. In particular, we
                        have that $u \leq h$ and
                        $\{u>0\}\subset \{h>0\}$ in $B_\rho$.
                        Moreover, since $h$ is the torsion
                          function multiplied by $\gamma_1$
                         with boundary datum $u$,
                        recalling that $u(0)=0$ and that $u$ is
                        Lipschitz continuous with constant
                        $L=L(R,\eta)$, see Lemma~\ref{le:lipschitz},
                        we deduce by a simple comparison argument that
			\begin{equation}\label{eq:Linftyh}
				\|h\|_{L^\infty(B_\rho)}\leq C_h \rho, 
			\end{equation}
			with a constant $C_h$ depending only on $R,\eta$.
			Thus, testing the optimality of $u$ with $h$, using also~\eqref{eq:propfeta} and an integration by parts, we have
			\begin{align*}
				\frac1\eta|B_\rho \cap \{u=0\}| &+\frac{q}{2}\int_{B_\rho}\int_{B_\rho}\frac{h^2(x)h^2(y)}{|x-y|}\,dxdy+M\left|\int_{B_\rho}(h^2-u^2)\,dx\right|\geq
				\int_{B_\rho}|\nabla u|^2\,dx-\int_{B_\rho}|\nabla h|^2\,dx  \\
				&=
				\int_{B_\rho}|\nabla (u-h)|^2\,dx +2 \int_{B_\rho} \Big(\nabla h\cdot \nabla (u-h)\Big)\, dx\\
				&=
				\int_{B_\rho}|\nabla (u-h)|^2\,dx +2 \int_{B_\rho} (-\Delta h) (u-h)\Big)\, dx+2\int_{B_\rho}(u-h)\frac{\partial h}{\partial \nu}\,d\mathcal H^2\\
				&=
				\int_{B_\rho}|\nabla (u-h)|^2\, dx-2\gamma_1\int_{B_\rho}(h-u)\,dx.
			\end{align*}
			Let us first treat the terms not involving the gradient: thanks to~\eqref{eq:Linftyh}, we can bound the term (using also the scaling of the Riesz energy and the fact that $\rho\leq 1$) 
			\begin{equation}\label{eq:estD}
				\frac{q}{2}\int_{B_\rho}\int_{B_\rho}\frac{h^2(x)h^2(y)}{|x-y|}\,dxdy\leq C_h^4\rho^4 \rho^{5}D(\chi_{B_1},\chi_{B_1})\frac{q}{2}\leq C_a |B_\rho| q,
			\end{equation}
			for a constant $C_a>0$, depending only on $R,\eta$.
			On the other hand, using again the bound $L^\infty$ on $h$ and $u$, we have 
			\begin{equation}\label{eq:estnormL1}
                          \int_{B_\rho}(h^2-u^2)\,dx+\int_{B_\rho}(h-u)\,dx\leq
                          (C_h+1) \int_{B_\rho}h\,dx\leq (C_h+1) C_h
                          \rho |B_\rho|=C_h' \rho|B_\rho|, 
			\end{equation}
			for a constant $C_h' > 0$ depending
                        only on $R,\eta$.
			
			Let us now focus on the gradient term.
			By the Poincar\'e and Cauchy-Schwarz inequalities, we have
			$$\int_{B_\rho}|\nabla (u-h)|^2\, dx \geq \frac{C_d}{|B_\rho|}\left(\frac1\rho \int_{B_\rho}(h-u)\,dx\right)^2,$$
			so in order to prove the upper bound in the claim, we first need to show that
			$\frac1{\rho|B_\rho|} \int_{B_{\rho}}(h-u)\,dx$ is bounded from below by a positive constant. Notice that, by the non-degeneracy of $u$ (see Remark~\ref{rmk:nondeg}), we have
			\[
			C \rho \leq \sup_{B_{\rho/2}} u \leq \sup_{B_{\rho/2}}h\,.
			\]
			On the other hand, since $h(x)+\frac{\gamma_1}{6}|x|^2$ is harmonic in $B_\rho$, the Harnack inequality in $B_\rho$ implies
			\[
			C \rho \le \sup_{B_{\rho/2}}h\le C_d\big(h(x)+\gamma_1\rho^2\big)\quad\text{for every}\quad x\in B_{\frac{\rho}2}\,.
			\]
			Thus, by taking $\rho_0$ such that $2C_d\rho_0\gamma_1\le C$, we obtain that $h \geq C_dC \rho = \overline C\rho$ in $B_{\frac{\rho}2}.$
			On the other hand, if $L=L(R,\eta)$ is the Lipschitz constant of $u$ (by Lemma~\ref{le:lipschitz}), then for any $\eps\in(0,1)$, $u\leq L \eps \rho$ in $B_{\eps \rho}$. Then
			$$\int_{B_\rho}(h-u)\,dx\geq\int_{B_{\eps \rho}} (h-u)\,dx\geq (\overline C \rho-L\eps \rho)|B_{\eps \rho}|,$$
			which, after choosing $\eps\le \frac12$ small enough, shows that \[
			\frac1{\rho} \int_{B_{\rho}}(h-u)\,dx\geq C_0|B_\rho|,
			\]
			for some constant $C_0>0$ depending only on $R,\eta$ and thus \[
			\int_{B_\rho}|\nabla (u-h)|^2\,dx\geq C_b|B_\rho|,
			\]
			for some $C_b>0$, depending only on $R,\eta$.

			At this point, using also~\eqref{eq:estnormL1} and~\eqref{eq:estD}, we have \[
			C_b|B_\rho|\leq \frac1\eta|B_\rho \cap \{u=0\}|+C_a q |B_\rho| + (M+2\gamma_1)\rho|B_\rho|\leq \frac1\eta|B_\rho \cap \{u=0\}|+C_a q |B_\rho| + C_g\rho|B_\rho|,
			\]
			for a universal constant $C_g>0$, recalling that $M$ is fixed (see Remark~\ref{rmk:M}) and that $\gamma_1$ is uniformly bounded by a constant depending only on $\|u\|_{L^\infty(B_R)}$, which in turn is uniform in $q$ (see Lemma~\ref{lem:inftybound}).
			It is then enough to take\[
			q_1\leq \min\Big\{q_0,\frac{C_b}{4C_a}\Big\},\qquad \rho_0\leq \frac{C_b}{4C_g},
			\]
			and we obtain that \[
			\frac{C_b}{4}|B_\rho|\leq \frac1\eta|B_\rho \cap \{u=0\}|,
			\]
			which entails the density estimate from above, so the claim is proved.
		}
	\end{proof}
	\begin{remark}
          Notice that the constants determining the Lipschitz
          regularity and the density estimates of the previous result
          do not depend on $\q$ for $q$ small enough.
	\end{remark}

        \noindent From now on, $q_1$ always refers to the constant
          defined in Lemma \ref{le:lemma4.11}.

	\subsection{Equivalence between the minimizations of $E_q$ and $\Eqme$}\label{sect:equivalence}
	In this section we show that unconstrained minima of $E_{q,M,\eta}$ and volume constrained minima of $E_q$ (or equivalently $E_{q,M}$) are actually the same. We begin by showing that
	for $\q$ small, the minimizers of $E_{q,M,\eta}$ in $B_R$ are close to a ball in $L^\infty$.
	To do that, we first start with an estimate that assures the $L^1-$proximity of an optimal set for problem~\eqref{eq:minGeta1} to a ball with radius not too large.

	\begin{lemma}\label{le:stimadiffsimm}
          Let $R>2$, $\q\in (0, q_1] $ and $\eta\in(0,1)$. Let
          $\Om=\Om_{q,M,\eta}$ be an optimal set
          for~\eqref{eq:minGeta1} such that $\Omega=\{u_\Omega>0\}$,
          where $u_\Omega\in H^1_0(\Omega)$ is the (positive) function
          attaining $E_q(\Omega)=E_{q,M}(\Omega)$, and
          $B=B_{q,M,\eta}$ a ball of measure $|\Omega|$ {attaining the
            Fraenkel asymmetry for $\Omega$, namely} such that
		\[
		\mathcal A (\Omega )=\frac{|\Omega\Delta B|}{|\Omega|}.
		\]
		Then, setting $u_B\in H^1_0(B)$ the function attaining $E_q(B)$,
		normalized so that $\|u_B\|_{L^2(B)}=\|u_\Om\|_{L^2(\Om)}=1$, we have
		\begin{equation}\label{eq:bounddiffsimm}
			|\{u_B>0\}\Delta \{u_\Omega>0\}|^2\leq {C q}{|\Omega|^{\frac{7}{3}}}. 
		\end{equation}
		for some universal constant $C>0$.
	\end{lemma}
	\begin{proof}
		Let $w_B$ be the {normalized} first eigenfunction of the Dirichlet Laplacian in $B$,
		and note that it is an admissible competitor for $E_q(B)$. {We note that $\{w_B>0\}=\{u_B>0\}=B$, see also Lemma~\ref{lem:cc2}.} Thanks
		to the quantitative Faber--Krahn inequality {(Theorem~\ref{thm:quantitativefk})}, we
		have
		\begin{equation}\label{eq:quant}
			\begin{split}
				|\Omega|^{2/3}\int_{\Omega}|\nabla u_\Omega|^2\, dx&\geq |\Omega|^{2/3}\lambda_0(\Omega)\\
				&\geq |\Omega|^{2/3}\lambda_0(B)+\widehat \sigma\frac{|\{u_B>0\}\Delta \{u_\Omega>0\}|^2}{|\Omega|^2}\\
				&=|\Omega|^{2/3}\int_{B} |\nabla w_B|^2\, dx+\widehat \sigma\frac{|\{u_B>0\}\Delta \{u_\Omega>0\}|^2}{|\Omega|^2}.
			\end{split}
		\end{equation}
		
		From the optimality of $\Omega$, we deduce, 
		\[
		\begin{split}
			\int_{\Omega}|\nabla u_\Omega|^2\, dx+\frac{\
				q}{2}\int_{\Omega}\int_{\Omega}&\frac{u_\Omega^2(x)
				u_\Omega^2(y)}{|x-y|}\,dx\,dy+f_\eta(|\Omega|)
			\leq E_q(B)+f_\eta(|B|)\\
			&\leq \int_{B}|\nabla w_B|^2\, dx+\frac{\
				q}{2}\int_{B}\int_{B}\frac{w_B^2(x)w_B^2(y)}{|x-y|}\,dx\,dy+f_\eta(|B|),
		\end{split}
		\]
		and using also~\eqref{eq:quant} we obtain\[
		\begin{split}
			\widehat \sigma\frac{|\{u_B>0\}\Delta \{u_\Omega>0\}|^2}{|\Omega|^2}&\leq |\Omega|^{2/3}\int_{\Omega}|\nabla u_\Omega|^2\, dx-|\Omega|^{2/3}\int_{B} |\nabla w_B|^2\, dx\\
			&\leq |\Omega|^{2/3}\frac{\q}{2}\Big(D(w_B^2,w_B^2)-D(u_\Omega^2,u_\Omega^2)\Big)\leq |\Omega|^{2/3}\frac{q}{2} D(w_B^2,w_B^2).
		\end{split}
		\]
{Now, using the $L^\infty$ bound on the first eigenfunction $w_B$ (see~\cite[Example~2.1.8]{davies}), and also the scaling of the Riesz functional, \[
		D(w_B^2,w_B^2)\leq C |\Omega|^{-2}\int_{ B} \int_{  B}\frac{dx\,dy}{|x-y|}\leq C|\Omega|^{-\frac{1}{3}},
		\]
for some universal constant $C>0$.}
		
		The previous two estimates lead to
		\begin{equation}\label{eq:stimaL1first}
			|\{u_B>0\}\Delta \{u_\Omega>0\}|^2\leq \frac{C}{\widehat \sigma}{|\Omega|^{\frac{7}{3}}}\frac{q}{2},
		\end{equation}
		{so the claim is proved.}
	\end{proof}

	A simple but important consequence of the previous result is the following lemma, stating that the
	measure of the ball $B=B_{q,M,\eta}$ {(to which any optimal set $\Omega_{q,M,\eta}$ is $L^1$-close)} is not too large, {as we show in the next result}. 
	\begin{lemma}\label{le:L1closetoball}
          Let $R>2$. There exists $\eta_1\in (0,1)$ such that for all
          $\q\in (0, q_1]$ and $\eta\leq \eta_1$, we have that
          any optimal set $\Om_{q,M,\eta}$ for
          problem~\eqref{eq:minGeta1}, such that
          $\Omega_{q,M,\eta}=\{u>0\}$, where $u\in H^1_0(\Omega)$ is
          {any} (positive) function attaining
          $E_q(\Omega)=E_{q,M}(\Omega)$, satisfies
		\begin{equation}\label{eq:stimaL1}
			|\Om_{{q,M,\eta}}|\leq |B_2|,\qquad |\Om_{{q,M,\eta}}\Delta B_{{q,M,\eta}}|\leq c_1\q,
		\end{equation}  
		{for some universal constant $c_1>0$.}
	\end{lemma}
	\begin{proof} 
		Let us suppose for the sake of contradiction
		that $|\Om_{q,M,\eta}|> |B_2|$. We are then going to reach a
		contradiction as long as\[ 1/\eta\ge E_{q_0}(B_1).
		\] 
		Since the functional\[
		q\mapsto E_{q,M,\eta}(\Om_{q,M,\eta}),
		\]
		is nondecreasing, we obtain  \[
		\sup_{\q\in(0,q_0)}E_{q,M,\eta}(\Om_{q,M,\eta})=E_{q_0,M,\eta}(\Om_{q_0,M,\eta})\leq E_{q_0,M,\eta}(B_1)=E_{q_0}(B_1),
		\]
		recalling that the optimal function for $E_{q_0}(B_1)$ is
		with unit $L^2$ norm.  On the other hand, using  the positivity of $E$, since $|\Om_{q,M,\eta}|> |B_2|$
		we have
		\[
		E_{q_0}(B_1)\geq E_{q_0,M,\eta}(\Om_{{q_0,M,\eta}})\geq
		\frac{1}{\eta}(|\Om_{{q_0,M,\eta}}|-|B_1|)\geq \frac{1}{\eta}
		(|B_2|-|B_1|).
		\]
		By choosing $\eta_1$ such that $\eta_1<1$ and
		\[
		\frac{(|B_2|-|B_1|)}{\eta_1}>E_{q_0}(B_1),
		\]
		we reach the desired contradiction.
		{The second part of the claim then follows from Lemma~\ref{le:stimadiffsimm}.}
	\end{proof}
	We note that in the above lemma, $\eta_1$ does not depend on $R$.
	Next we show that, for $\q$ small, the boundary of any optimizer
	$\Om_{{q,M,\eta}}$ (such that $\Omega_{ q,M,\eta}=\{u>0\}$, where
	$u\in H^1_0(\Omega)$ is the (positive) function attaining
	$E_q(\Omega)=E_{q,M}(\Omega)$) is close to the one of the
	corresponding optimal ball of the same measure $B_{{q,M,\eta}}$ in the
	definition of Fraenkel asymmetry, with respect to the Hausdorff
	distance $d_H$ (see \cite[Definition 4.4.9]{amti} for more details
	about the Hausdorff distance).
	\begin{lemma}\label{le:closehausdorff}
          Under the assumptions of Lemma~\ref{le:L1closetoball}, for
          all $\delta>0$ there exists
          $q_\delta=q_\delta(R,\eta)\in(0,\q_1 ]$ such that for
          all $\q\leq q_\delta$, we have
          \[ {\rm dist}_H(\partial \Om_{{q,M,\eta}},\partial
            B_{{q,M,\eta}})\leq \delta.
		\] 
	\end{lemma}
	\begin{proof}
		This follows exactly as in \cite[Lemma~5.4]{maru}, we report here the proof for the sake of completeness.
		{
			We fix $\delta>0$ and call $B_\delta(B_{{q,M,\eta}}):=B_{{q,M,\eta}}+B_\delta$ the $\delta$-neighborhood of $B_{{q,M,\eta}}$. First of all, we aim to prove that $\Om_{{q,M,\eta}}\subset B_\delta(B_{{q,M,\eta}})$.
			If $\Om_{{q,M,\eta}}\setminus B_\delta(B_{{q,M,\eta}})$ is empty, then there is nothing to prove. Otherwise there exists $x\in \Om_{{q,M,\eta}}\setminus B_\delta(B_{{q,M,\eta}})$ so that by Lemma~\ref{le:lemma4.11} there exists $\rho_0(R,\eta)$ such that for   $\rho\leq \rho_1:=\min\{\rho_0(R,\eta),\delta\}$ it holds
			\[
			|B_1| \theta\rho^3\leq |B_\rho(x)\cap \Om_{{q,M,\eta}}|\leq |\Om_{{q,M,\eta}}\setminus B_{{q,M,\eta}}|\leq c_1 q,
			\] 
			where the last estimate follows from Lemma~\ref{le:L1closetoball} and precisely~\eqref{eq:stimaL1}.
			Notice that the choice of $\rho_1\leq \delta$ assures that $|B_\rho(x)\cap \Om_{{q,M,\eta}}|\leq |\Om_{{q,M,\eta}}\setminus B_{{q,M,\eta}}|$. In conclusion, choosing $\rho=\rho_1$, we have \[
			|B_1|\theta\rho_1^3\leq c_1 q,
			\]
			which is not possible as soon as \[
			q\leq q_\delta:=\frac{|B_1|\theta}{c_1}\rho_1^3.
			\]
			With the same argument, thanks to the outer density estimate from Lemma~\ref{le:lemma4.11} and again to the $L^1$ proximity from Lemma~\ref{le:L1closetoball}, we can show also that  $B_R\setminus \Omega_{q,M,\eta}\subset B_\delta(B_R\setminus B_{{q,M,\eta}})$, with the same notation as above. This concludes the proof.
		}
	\end{proof}
	\begin{remark}
		It is worth noting that the constant $\q_\delta$ in the lemma above depends also on $R$. This is one of the main obstacles while trying to get rid of the equiboundedness assumption of Theorem~\ref{thm:mainbdd}. 
	\end{remark}
	
\begin{remark}\label{rmk:closehausdorff}
		In view of the previous result, {first of all we not that the energy is invariant under translations and thus we can assume that $\Omega\subset B_3$}, then we fix $q_2(R):=\min\{q_1,q_\delta\}$, where $q_\delta$ is the constant from Lemma~\ref{le:closehausdorff} with the choice of $\delta:=1/2$.
		
		Therefore, if $\q \in (0, q_2]$, in the proof of the
                next Theorem~\ref{thm:noconstraint}, we are allowed to
                inflate a set without touching the boundary of the
                geometric constraint $\Omega\subset B_R$.
	\end{remark}

	We can now show the equivalence between the constrained and the unconstrained problems. {We recall that the constant $M$ has been already fixed, see Remark~\ref{rmk:M}.}
	
	\begin{theorem}\label{thm:noconstraint}
          There exists a universal constant $R_0\geq 10$ such that,
          for all $R\geq R_0$, there exists $ q_3= q_3(R)\leq q_2$ and
          $\eta_2= \eta_2(R)\leq \eta_1$ such that, for all
          $\eta\leq \eta_2$ and $\q \in (0, q_3]$, we have that
		\begin{equation}\label{eq:nocontraint}
			\begin{aligned}
				\min\left\{E_{\q,M,\eta}(\Om) : \Om\subset B_R \right\} {=}
				\inf\left\{E_q(\Om) : \Om\subset B_R,\; |\Om|=|B_1|\right\}.
			\end{aligned}
		\end{equation}
		As a consequence, problems \eqref{eq:mainpb} and \eqref{eq:minGeta1}
		are equivalent {for these values of $q$ and $\eta$}.
	\end{theorem}
	\begin{proof}
		It is easy to check that \[
		\begin{split}
			\min\left\{E_{\q,M,\eta}(\Om) : \Om\subset B_R\right\}
			\le\inf\left\{ E_q(\Om) : \Om\subset B_R,\; |\Om|=|B_1|\right\}
			{=:\mu(\q, R)},
		\end{split}
		\]
		as the two functionals coincide on sets of measure $|B_1|$, thanks to
		the definition of $f_\eta$.  Then, if the {reversed inequality}
		holds, it follows that on the set of minimizers (of the first or of
		the second problem) the two functionals do coincide, that is, problems
		\eqref{eq:mainpb} and \eqref{eq:minGeta1} are equivalent.
		
		We prove the claim of the theorem by
		contradiction. Let {$R>R_0$ to be chosen later and}
		\[
		\Om_{{q,M,\eta}}\subset B_R,\quad \sigma_{{q,M,\eta}}\in\R,\quad |\Om_{{q,M,\eta}}|=|B_1|+\sigma_{{q,M,\eta}},\quad E_{q,M,\eta}(\Om_{{q,M,\eta}})<\mu,
		\]
		and we also note that, $\mu\leq E_q(B_1)$, by definition of infimum.
		We moreover assume, without loss of generality, that $\Om_{{q,M,\eta}}$ are minimizers for problem~\eqref{eq:minGeta1}.
		We treat separately the case $\sigma_{{q,M,\eta}}>0$ and $\sigma_{{q,M,\eta}}<0$.
		
		\par
		{\bf Case $\sigma_{{q,M,\eta}}>0$. }
		We first observe that $\sigma_{{q,M,\eta}}\to0$ as $\eta\to0$. Indeed (recalling also that $E_{q}=E_{q,M}$ thanks to Proposition~\ref{prop:vincoloL2})
		\[
		E_{q,M,\eta}(\Om_{{q,M,\eta}})=E_{q}(\Om_{{q,M,\eta}})+\frac{1}{\eta}\sigma_{{q,M,\eta}}
		\]
		and so 
		\[
		0\le\frac{\sigma_{q,M,\eta}}{\eta}= E_{q,M,\eta}(\Om_{q,M,\eta})-E_{q}(\Omega_{{q,M,\eta}})\le E_q(B_1),
		\]
		using the assumption
		$E_{q,M,\eta}(\Om_{q,M,\eta})\leq \mu\leq E_q(B_1)$ and the positivity
		of {the energy}. This implies that $\sigma_{{q,M,\eta}}\to0$ as
		$\eta\to0$.
		
		Let now ${\rho_{{q,M,\eta}}}<1$ be such that
		$|{\rho_{{q,M,\eta}}}\Om_{{q,M,\eta}}|=|B_1|$, therefore 
		\[
		{ {\rho_{{q,M,\eta}}}= 1
			-\frac{\sigma_{{q,M,\eta}}}{3|B_1|}+C\sigma_{{q,M,\eta}}^2,}
		\]
		for some {$C = C(\sigma_{q,M,\eta}) \in \R$ such
                  that $|C| \leq C_0$ for all
                  $|\sigma_{q,M,\eta}| < \tfrac12 |B_1|$ some
                  $C_0 > 0$ universal}.
		
		We call $u=u_{q,M,\eta}$ an optimal normalized function attaining $E_q(\Omega_{q,M,\eta})$, thus the function\[
		\widetilde u(y)={\rho_{{q,M,\eta}}}^{-\frac{3}{2}}u\Big(\frac{y}{{\rho_{{q,M,\eta}}}}\Big),\qquad y\in {\rho_{{q,M,\eta}}}\Omega_{q,M,\eta},
		\]
		is an admissible  competitor with unitary $L^2-$norm for $E_q({\rho_{{q,M,\eta}}}\Om_{q,M,\eta})$.
		We  have the following scalings
		\begin{equation}\label{eq:split}
			\begin{split}
				\int_{{\rho_{{q,M,\eta}}}\Omega_{q,M,\eta}} |\nabla \widetilde u(y)|^2\,dy&={\rho_{{q,M,\eta}}}^{-2}\int_{\Om_{q,M,\eta}}|\nabla u(x)|^2\,dx,\\
				D(\widetilde u^2,\widetilde u^2)&=\int_{{\rho_{{q,M,\eta}}}\Omega_{q,M,\eta}}\int_{{\rho_{{q,M,\eta}}}\Omega_{q,M,\eta}}\frac{\widetilde u^2(x)\widetilde u^2(y)}{|x-y|}\,dxdy\\
				&={\rho_{{q,M,\eta}}}^{-1}\int_{\Omega_{q,M,\eta}}\int_{\Omega_{q,M,\eta}}\frac{ u^2(w) u^2(z)}{|w-z|}\,dwdz\\
				&={{\rho_{{q,M,\eta}}}^{-1}}D(u^2,u^2).
			\end{split}
		\end{equation}
		Since the new set $\rho_{{q,M,\eta}}\Om_{{q,M,\eta}}$ is now admissible in the constrained minimization problem~\eqref{eq:mainpb}, using the  above scaling we obtain
		\[
		\begin{split}
			E_{q,M,\eta}(\Om_{q,M,\eta})&=E_q(\Om_{{q,M,\eta}})+\frac{\sigma_{{q,M,\eta}}}{\eta}\\
			&<\mu\\&\leq E_q({\rho_{{q,M,\eta}}}\Om_{{q,M,\eta}})\\
			&\leq \int_{{\rho_{{q,M,\eta}}}\Omega_{q,M,\eta}}|\nabla \widetilde u(y)|^2\,dy+\frac{\ q}{2}D(\widetilde u^2,\widetilde u^2)\\&={\rho_{{q,M,\eta}}}^{-2}\int_{\Om_{q,M,\eta}}|\nabla u(x)|^2\,dx+{\rho_{{q,M,\eta}}}^{-1}\frac{\ q}{2} D(u^2,u^2)\\
			&=\int_{\Om_{q,M,\eta}}|\nabla u(x)|^2\,dx \left(1+{\frac{2\sigma_{{q,M,\eta}}}{3|B_1|}}+C\sigma_{{q,M,\eta}}^2\right)\\
			&\,\,\,\,+\frac{\ q}{2}  D(u^2,u^2)\left(1+{\frac{\sigma_{{q,M,\eta}}}{3|B_1|}}+C\sigma_{{q,M,\eta}}^2\right),
		\end{split}
		\]
		we deduce that (up to increasing $C$, recalling also that
		$E_q(\Omega_{q,M,\eta})\leq E_q(B_1)$)
		\[
		\begin{aligned}
			\frac{\sigma_{{q,M,\eta}}}{\eta}&<\int_{\Om_{q,M,\eta}}|\nabla u(x)|^2\,dx \left({\frac{2\sigma_{{q,M,\eta}}}{3|B_1|}}\right)+\frac{\ q}{2}  D(u^2,u^2)\left({\frac{\sigma_{{q,M,\eta}}}{3|B_1|}}\right)+ 2E_q(B_1)C\sigma_{{q,M,\eta}}^2\\
			&\le {\frac{\sigma_{q,M,\eta}}{3|B_1|}} 2 E_q(\Omega_{q,M,\eta})+C\sigma_{{q,M,\eta}}^2.
		\end{aligned}
		\]
		
		Thus, for some universal $C>0$,
		\[
		\frac{1}{\eta}\le C E_q(\Omega_{q,M,\eta})+C\sigma \le C E_q(B_1),
		\]
		which leads to a contradiction as soon as $\eta_2<\frac{1}{C (E_q(B_1))}.$

		\par
		{\bf Case $\sigma_{{q,M,\eta}}<0$. } For this case let us call \[
		\rho_{{q,M,\eta}}:=\Big(1+\frac{\sigma_{{q,M,\eta}}}{|B_1|}\Big)^{-1/3},
		\]
		so that $|\rho_{{q,M,\eta}}\Om_{{q,M,\eta}}|=|B_1|$.

		We recall from the previous sections that a minimizer
		$\Om_{{q,M,\eta}}$ for $E_{q,M,\eta}$ exists, and by
		Lemma~\ref{le:closehausdorff}, up to taking $\q_3\leq \q_2$ as in
		Remark~\ref{rmk:closehausdorff}, and $\eta_2<\eta_1$ as in
		Lemma~\ref{le:L1closetoball}, the rescaled set
		$\rho_{{q,M,\eta}}\Om_{{q,M,\eta}}$ is still contained in $B_R$, as
		soon as $R_0$ is big enough.

		In fact, we show that $\sigma_{q,M,\eta}\geq -\frac34|B_1|$, thus $\rho_{q,M,\eta}\leq 4^3$ (hence, recalling Remark~\ref{rmk:closehausdorff} we can take any $R_0> 2\cdot 4^3$). 
		
		If, for the sake of contradiction, $\sigma_{q,M,\eta}<-\frac34|B_1|$, then $|\Omega_{q,M,\eta}|\leq\frac14|B_1|$ and $|2^{1/3}\Omega_{q,M,\eta}|\leq \frac12|B_1|$.
		The optimality of $\Omega_{q,M,\eta}$ entails, 
		\[
		\begin{split}
			E_{q,M,\eta}(\Omega_{q,M,\eta})&=E_q(\Omega_{q,M,\eta})+\eta\sigma_{q,M,\eta}\leq E_q(2^{1/3}\Omega_{q,M,\eta})+\eta(2|\Omega_{q,M,\eta}|-|B_1|)\\
			&=E_q(2^{1/3}\Omega_{q,M,\eta})+\eta(|B_1|+2\sigma_{q,M,\eta}),
		\end{split}
		\]
		which is equivalent to say\[
		E_q(\Omega_{q,M,\eta})-E_q(2^{1/3}\Omega_{q,M,\eta})\leq \eta(|B_1|+\sigma_{q,M,\eta}).
		\]
                {On the other hand (calling $u_\Omega$ the $L^2$
                  normalized function attaining
                  $E_q(\Omega_{q,M,\eta})$) we have the upper bound on
                  $D(\cdot,\cdot)$ thanks to Lemma~\ref{le:hardytype}
                  (for a universal constant
                  $C>0$)\[ D(u_\Omega^2,u_\Omega^2)\leq {2
                      \|u_\Omega\|_{L^2}\|\nabla u_\Omega\|_{L^2}}
                    \int_{\Omega_{q,M,\eta}} u_\Omega^2 \,dx\leq
                    C\|\nabla u_\Omega\|_{L^2}.
\]
Then, thanks to the above estimate and using also the positivity of
$D(\cdot, \cdot)$, the Faber-Krahn inequality and up to a decrease of
$q_3$\[
\begin{split}
		E_q(\Omega_{q,M,\eta})-E_q(2^{1/3}\Omega_{q,M,\eta})&= (1-2^{-2/3})\int_{\Omega_{q,M,\eta}} |\nabla u_\Omega|^2\,dx+(1-2^2)\frac{q}{2}\int_{\Om_{{q,M,\eta}}}\int_{\Omega_{q,M,\eta}}\frac{u_\Omega^2(x) u_\Omega^2(y)}{|x-y|}\,dxdy\\
		&\geq (1-2^{-2/3})\int_{\Omega_{q,M,\eta}} |\nabla u_\Omega|^2\,dx-2C q\left(\int_{\Omega_{q,M,\eta}}|\nabla u_\Omega|^2\,dx\right)^{1/2}\\
				&\geq (1-2^{-2/3})\int_{\Omega_{q,M,\eta}} |\nabla u_\Omega|^2\,dx-2C q\max\left\{1,\left(\int_{\Omega_{q,M,\eta}}|\nabla u_\Omega|^2\,dx\right)^{1/2}\right\}\\
		&\geq \frac{(1-2^{-2/3})}{2}\lambda_0(\Omega_{q,M,\eta})\geq \frac{(1-2^{-2/3})}{2}\lambda_0(B_1).
\end{split}
\]
Finally, $(|B_1|+\sigma_{q,M,\eta})\leq |B_1|$, so we reach a contradiction as soon as $\eta\leq\eta_2$ and \[
		\eta_2<\frac{(1-2^{-\frac{2}{3}})}{2}\frac{\lambda_0(B_1)}{|B_1|}.
		\]
}		
		Let us define the function, using the same notations for $u,\widetilde u$ as in the previous case,\[
		g\colon[1,\rho_{{q,M,\eta}}]\rightarrow \R,\qquad g(r)=\int_{r\Om_{q,M,\eta}}|\nabla \widetilde u|^2\, dx+\frac{\ q}{2} \int_{r\Om_{{q,M,\eta}}}\int_{r\Omega_{q,M,\eta}}\frac{\widetilde u^2(x)\widetilde u^2(y)}{|x-y|}\,dxdy+\eta(r^3|\Om_{{q,M,\eta}}|-|B_1|).
		\]
We show that the minimum of the function $g$ is attained at $r=\rho:=\rho_{{q,M,\eta}}$. 
{Then the proof is concluded because this implies that $E_q(\rho_{q,M,\eta}\Omega_{q,M,\eta})=E_{q,M,\eta}(\rho_{q,M,\eta}\Omega_{q,M,\eta})\leq E_{q,M,\eta}(\Omega_{q,M,\eta})\leq E_q(\Omega')$ for all $\Omega'\subset B_R$ with $|\Omega'|=|B_1|$.}		
This is equivalent to show that for some $\eta$ the inequality  
		\[
		g(r)\geq \int_{\rho\Om_{q,M,\eta}}|\nabla \widetilde u|^2\, dx+\frac{\ q}{2} \int_{\rho\Om_{{q,M,\eta}}}\int_{\rho\Omega_{q,M,\eta}}\frac{\widetilde u^2(x)\widetilde u^2(y)}{|x-y|}\,dxdy,\qquad \text{for all }r\in[1,\rho],
		\]
		holds true. Up to rearranging the terms, and by the rescaling of the involved integrals, such an inequality reads as
		\[
		\eta\left(1-\left(\frac{r}{\rho}\right)^3\right)\le \int_{\rho\Om_{q,M,\eta}}|\nabla \widetilde u|^2\left(\left(\frac{r}{\rho}\right)^{-2}-1\right)+\frac{\ q}{2} D(\widetilde u^2,\widetilde u^2)\left(\left(\frac{r}{\rho}\right)^{-1}-1\right).
		\]
		Setting $t:=\frac r\rho < 1$, and observing that $r^3|\Om_{{q,M,\eta}}|=t^3$, the last inequality is equivalent to
		\[
		\eta\le \frac{\int_{\rho\Om_{q,M,\eta}}|\nabla \widetilde u|^2\, dx(t^{-2}-1)+\frac{\ q}{2} D(\widetilde u^2,\widetilde u^2)(t^{-1}-1)}{1-t^3}.
		\]
		It is easy to check that the right hand side is bounded from below by the function \[
		t\mapsto\lambda_0(B)\frac{t^{-2}-1}{1-t^3},\qquad t\in(0,1),
		\]
		which is a function strictly decreasing in its domain and with infimum given by \[
		\lim_{t\to 1^-}\lambda_0(B)\frac{t^{-2}-1}{1-t^3}=\frac{2}{3}>0.
		\]
		
		Thus it is enough to take $\eta\leq \eta_2\leq 2/3$ and we immediately deduce that $g$ has minimum for $r=\rho$.
		This concludes the proof.
	\end{proof}

	We highlight that if $\eta>0$ 
	is such that Theorem~\ref{thm:noconstraint} holds true, we can
        freely assume the equivalence of the volume-constrained
        minimization of $E_q$ and the unconstrained one for
        $E_{q,M,\eta}$ (since $M$ has been already fixed, see
          Remark~\ref{rmk:M}).
	\noindent
	On the other hand, we stress that this choice of $\eta$ depends, in all our estimates, on $R$.

	\section{Optimality conditions and improvement of flatness}\label{sec:highreg}
	
	We have now the following picture. 
	We know that minimizers for the auxiliary functional $E_{q,M,\eta}$ exist and (with the choices we made for $M,\eta$) are the same of those of the volume constrained functianal $E_q$.
	These minimizers satisfy density estimates {which are uniform with respect to $q\leq q_3$ (see~Lemma \ref{le:lemma4.11} for the density estimate and Theorem~\ref{thm:noconstraint} for the constant $q_3$). }
	Moreover in Lemma \ref{le:closehausdorff} we have shown that such minimizers
	are close in Hausdorff distance to a given ball (any ball achieving
	the minimum in the definition of {Fraenkel} asymmetry of
	$\Omega$). We now improve such a regularity with the final scope of
	showing that optimal sets are {uniform} $C^{2,\alpha}$
	parametrizations on the boundary of the ball. The results and the
	proofs in this section are {borrowed} (with nontrivial
	adjustments) from results in \cite[Theorem~4.5 and Theorem~4.8]{alca},
	\cite[Theorem~2]{aac}. We begin with the following theorem, {in which we use the notation $\partial^*F$ for the reduced boundary of a set of finite perimeter $F$}.
	\begin{theorem}\label{thm:agalcathm2}
          Let $q \in (0,q_3 ]$, {let} $\Omega$ be a minimizer of
          \eqref{eq:minmin1}, and {let} $u$ be an optimal function
          attaining $E_q(\Omega)$, thus also solution of
          \eqref{eq:ellpde}. Then we have that:
		\begin{enumerate}
			\item[(i)] There is a Borel function $\mu_{u}\colon \partial \Om \rightarrow \R$ such that, in the sense of the distributions, one has
			\begin{equation}\label{eq:bdp4.31}
				-\Delta u = (\lambda_q- q v_u) u - \mu_{u}\mathcal H^{2}\resmeas\partial\Omega,\qquad\text{ in }B_R.
			\end{equation}
			\item[(ii)]  There exist constants $0<c<C<+\infty$, depending on $R$,  such that $c\le \mu_{u}\le C$.
			\item[(iii)] For all points $\overline x\in \partial^*\Om =\partial^*\{u>0\}$, the measure theoretic inner unit normal $\nu_{u}(\overline x)$ is well defined and, as $\rho\to0$, 
			\begin{equation}\label{eq:bdp4.32}
				\frac{\Om-\overline x}{\rho}\rightarrow \{x : x\cdot \nu_{u}(\overline x)\geq 0\},\qquad \text{in }L^1(B_R).
			\end{equation}
			\item[(iv)]For $\mathcal H^{2}$ almost all $\overline x\in \partial^*\{u>0\}$ we have 
			\begin{equation}\label{eq:bdp4.33}
				\frac{u(\overline x+\rho x)}{\rho}\longrightarrow \mu_{u}(\overline x)(x\cdot \nu_{u}(\overline x))_+,\qquad \text{in }W^{1,p}(B_R)\;\text{for every }p\in[1,+\infty).
			\end{equation}
			\item[(v)] $\mathcal H^{2}(\partial \Om\setminus\partial^*\Om)=0$.
		\end{enumerate}
	\end{theorem}
	
	\begin{proof}
		The proof is essentially identical to that in \cite[Section
		4]{alca}. We only have to check that our hypotheses match with those
		in \cite{alca}.  First by Lemma \ref{lem:inftybound} $u$ satisfies
		\[
		-\Delta u-Q(x)=0 {\qquad \text{in} \ \mathcal D'(\Omega),} 
		\]
		where
		$Q=(\lambda_q- q v_u) u\in L^\infty(\Omega)$ and
		$u\in H^1_0(\Omega)$. Hence, by repeating the proof of
		\cite[Theorem 4.5]{alca} or by directly applying
		\cite[Proposition 2.3]{bmpv} one obtains that there exists a
		positive Radon measure concentrated on
		$\partial\Omega$ that we denote $\mu_u\mathcal H^2\resmeas\partial\Omega$. Moreover, {thanks to the non-degeneracy, see Remark~\ref{rmk:nondeg} and the Lipschtiz continuity of $u$} we have that there exist constant $C>c>0$
		depending on $q$ and $R$ such that 
		\[
		c\le{\frac1r \mean{\partial B_r}{u\,d\mathcal H^2}}\le C.
		\]
		Hence we can work under the hypotheses of \cite[Theorem 4.5]{alca} so that $\mu_u$ is a density of a Radon measure on $\partial\Omega$ and, denoting still with $\mu_u$ the function defininig it, $\mu_u$ satisfies $(i)-(v)$.
	\end{proof}

	\subsection{The structure of $\mu_u$: blow up limits}
	We show now the following result.
	\begin{proposition}
		Let $\Omega$ be a minimizer of \eqref{eq:mainpb}. Then  function $\mu_u:\partial \Omega\to \R$ found in  Theorem \ref{thm:agalcathm2} is constant on $\partial^*\Omega$.
	\end{proposition}
	\begin{proof}
		The proof follows the path of \cite[Theorem 6.5]{maru}, in turn inspired by \cite{aac}. 
		Due to the nonlocal term, we will have to perform some new and non-straightforward computations.
		
		We reason by contradiction and we assume that  there exists $x_0,x_1\in\partial^*\Omega$ such that
		\[
		\mu_u(x_0) < \mu_u(x_1).
		\] 
		Then we construct a family of volume preserving diffeomorphisms  as follows: let $\kappa<1$ and $\rho<1$  and let $\varphi\in C^1_0(B_1(0))$ be a non-null, radially symmetric function supported in $B_1(0)$. We define 
		\[
		\tau_{\rho,\kappa}(x)=\tau(x)=
		x+\sum_{i\in\{0,1\}}(-1)^i\kappa\rho\varphi\left(\frac{|x-x_i|}{\rho}\right)\nu_{x_i} \chi_{B_\rho(x_i)},
		\]
		where $\nu_{x_i}$ are the measure theoretic inner normals to $\partial^*\Omega$ at $x_i$, $i=1,2$.
		
		It is easy to notice that $\tau$ is indeed a diffeomorphism for $\rho$ and $\kappa$ small enough and that   $\tau(x)-x$ vanishes outside $B_\rho(x_0)\cup B_\rho(x_1)$. Moreover we have:
		\begin{equation}\label{eq:dettau}
			\nabla\tau(x)=Id+\sum_{i\in\{0,1\}}(-1)^i\kappa\varphi'\left(\frac{|x-x_i|}{\rho}\right)\frac{x-x_i}{|x-x_i|}\otimes\nu_{x_i} \chi_{B_\rho(x_i)}, 
		\end{equation}
		so that\footnote{We are using the formula $\det(Id+\xi A)=1+   trace(A)\xi+o(\xi)$ for a matrix $A\in\R^{N\times N}$.}
		\begin{equation}\label{eq:espdet}
			\det(\nabla\tau(x)) = 1+ \sum_{i\in\{0,1\}}(-1)^i\kappa\varphi'\left(\frac{|x-x_i|}{\rho}\right)\frac{x-x_i}{|x-x_i|}\cdot\nu_{x_i} \chi_{B_\rho(x_i)}+o(\kappa).
		\end{equation}
		We call $\Omega_\rho =\tau(\Omega)$. We aim to show that for  $\kappa,\rho$ small enough it holds $E_{q,M,\eta}(\Omega_\rho)< E_{q,M,\eta}(\Omega)$, hence contradicting the minimality of $\Omega$. To do that, we deal with the first variation of each term of the sum defining $E_{q,M,\eta}$. We stress that the computations regarding the volume and the Dirichlet energy contributions are identical to those performed originally in \cite{aac} (see also \cite{brdeve} and \cite{demamu}, where the same idea is applied). Moreover, exactly as in the proof of \cite[Theorem 6.5]{maru} one obtains that 
		\begin{equation}\label{variazionevolume}
			f_\eta(\Omega_\rho)-f_\eta(\Omega)= o(\rho^3),\qquad \text{as }\rho\rightarrow 0,
		\end{equation}
		and that
		\begin{equation}\label{variazionetorsione}
			\frac{1}{\rho^3} \Big(\int_{\Omega_\rho}|\nabla u_\rho|^2\, dx-\int_\Omega |\nabla u|^2\, dx\Big)\le\kappa{ (\mu_u^2(x_0)-\mu_u^2(x_1))}C(\varphi) + o_\rho(1)+o(\kappa),
		\end{equation}
		where $u_\rho$ and $u$ are the functions attaining $E_q(\Omega_\rho)$ and $E_q(\Omega)$ respectively, and 
		\begin{equation}\label{rfk}
			C(\varphi)=\int_{B_1(0)\cap\{y\cdot\nu=0\}}\varphi(|y|)\,d\mathcal H^{2}(y)=-\int_{B_1(0)\cap \{y\cdot \nu>0\}}\varphi'(|y|)\frac{y\cdot \nu}{|y|}\,dy,
		\end{equation}
		with the last equality that follows from the Divergence Theorem, recalling that $\nu$ is a inner normal and $${\mathrm{div}}(\varphi(|y|)\nu)=\varphi'(|y|)\frac{y\cdot \nu}{|y|}.$$
		Notice also that by the radial symmetry of $\varphi$  the value of $C(\varphi)$ is not affected by the choice of $\nu$.
		
		We are left to compute the variation of the nonlocal term $D(\cdot,\cdot)$. This is the major technical difference with respect to the proof of \cite[Theorem 6.5]{maru}.
		We claim that (recalling that $u_\rho$ and $u$ are the functions attaining $E_q(\Omega_\rho)$ and $E_q(\Omega)$ respectively)
		\begin{equation}\label{eq:claimh}
			\frac{1}{\rho^3}\left(D(u_{\rho}^2,u_\rho^2)-D(u^2,u^2)\right)=o(\kappa)+o(\rho).
		\end{equation}
		Once that~\eqref{eq:claimh} is proved, the conclusion then readily follows: by minimality of $\Omega$ and thanks to \eqref{variazionevolume}, \eqref{variazionetorsione} and \eqref{eq:claimh} we have that
		\[
		\begin{split}
			&0\le E_{q,M,\eta}(\Omega_\rho)-E_{q,M,\eta}(\Omega)\\
			&{ \le \kappa\rho^3C(\varphi)\Big((\mu_u(x_0)^2-\mu_u(x_1)^2) \Big)+o(\rho^3)+\rho^3 o(\kappa).}
		\end{split}
		\]
		Since { from the assumptions there holds $\mu_u(x_0)^2-\mu_u(x_1)^2<0$} we get the desired contradiction  by choosing { $\rho$ and $\kappa$ small enough}.
		
		It remains to show the validity of \eqref{eq:claimh}. To do so, we set
		\[
		\wu(x)=u(\tau^{-1}(x)) \quad\text{ and }\quad \ww(x)=v_{\wu}(x)\wu(x)^2,
		\]	
		where $v_u(x)=\int_\Omega \frac{u^2(y)}{|x-y|}dy$. With such a notation in force we compute, using also formula \eqref{eq:dettau},
		\begin{equation}\label{eq:h0}
			\begin{aligned}
				\frac{1}{\rho^3}\left( D(u_{\rho}^2,u_\rho^2)-D(u^2,u^2)\right)&=\frac{1}{\rho^3}\left(\int_{\Omega_\rho}\ww\, dx-\int_\Omega w \, dx\right)\\
				&=\frac{1}{\rho^3} \int_\Omega \left(\ww(\tau(x))\det(\nabla\tau(x))-w(x) \right)\,dx\\
				&=\frac{1}{\rho^3}\int_\Omega(\ww(\tau(x))-w(x))\,dx  \\
				&+\frac{1}{\rho^3}\int_\Omega\ww(\tau(x))\sum_{i\in\{0,1\}}(-1)^i\kappa\varphi'\left(\frac{|x-x_i|}{\rho}\right)\frac{x-x_i}{|x-x_i|}\cdot\nu_{x_i} \chi_{B_\rho(x_i)}\,dx+o(\kappa).				
			\end{aligned}
		\end{equation}	
		
		We observe that since $\ww(x)=v_{\wu}(x)\wu(x)^2$, with $v_{\wu}$ uniformly bounded and $\wu$ Lipschitz continuous in $\Omega$, then $|\widetilde w (\tau(x))|\leq C \rho^2$ in $\Omega\cap B_\rho(x_i)$, {since} $\wu(x_i)=0$. With this in mind, we can compute
{		\begin{equation}\label{eq:h1}
			\begin{aligned}
				\Big|\frac{1}{\rho^3}&\int_\Omega\ww(\tau(x))\sum_{i\in\{0,1\}}(-1)^i\kappa\varphi'\left(\frac{|x-x_i|}{\rho}\right)\frac{x-x_i}{|x-x_i|}\cdot\nu_{x_i} \chi_{B_\rho(x_i)}\,dx \Big|\\
				&\leq  \frac{1}{\rho^3}\sum_{i\in\{0,1\}}\kappa\int_{\Omega\cap B_\rho(x_i)} \Big|\ww(\tau(x))\Big|\cdot \Big|\varphi'\left(\frac{|x-x_i|}{\rho}\right)\frac{x-x_i}{|x-x_i|}\cdot\nu_{x_i}\Big| \,dx\\
&\leq \frac{C}{\rho}|B_\rho|=o(\rho).
			\end{aligned}
		\end{equation}
}		
		Moreover
		\begin{equation}\label{eq:hint}
			\begin{aligned}
				\frac{1}{\rho^3}\int_\Omega(\ww(\tau(x))-w(x))\,dx&=\frac{1}{\rho^3}\int_\Omega u^2(x)\left(\int_{\Omega_\rho}\frac{\wu^2(y)}{|\tau(x)-y|}\,dy-\int_\Omega\frac{u^2(y)}{|x-y|}\,dy \right)\,dx\\
				&=\frac{1}{\rho^3}\int_\Omega u^2(x)\left(\int_{\Omega_\rho}\frac{u^2(\tau^{-1}y)}{|\tau(x)-y|}\,dy-\int_\Omega\frac{u^2(y)}{|x-y|}\,dy \right)\,dx\\
				&=\frac{1}{\rho^3}\int_\Omega u^2(x)\int_{\Omega}\left( \frac{u^2(y)}{|\tau(x)-\tau(y)|}\det(\nabla\tau(y))-\frac{u^2(y)}{|x-y|}\right)\,dy\,dx
			\end{aligned}
		\end{equation}
		and, with a computation similar to the one done in \eqref{eq:h1} we obtain that
		\begin{equation}
			\frac{1}{\rho^3}\int_\Omega(\ww(\tau(x))-w(x))\,dx=\frac{1}{\rho^3}\int_\Omega u^2(x)\int_{\Omega\cap(B_\rho(x_0)\cup B_\rho(x_1))} u^2(y){\left(\frac{1}{|\tau(x)-\tau(y)|}-\frac{1}{|x-y|}\right)}\,dy\,dx{+o(\rho)}.
		\end{equation}
		
		Now, by Lemma~\ref{le:lipschitz}, we know that $u$ is Lipschitz {(with constant $L$)} so that $u^2\lesssim \rho^2$ in $\Omega \cap B_\rho(x_i)$, for $i=1,2$, since $u(x_i)=0$. Hence we obtain by the previous formula and an elementary computation that
		\begin{equation}\label{eq:h2}
			\begin{aligned}
				\Big|\frac{1}{\rho^3}\int_\Omega(\ww(\tau(x))-w(x))\,dx\Big|&\le \frac{{L}}{{\rho}}\int_\Omega u^2(x)\int_{\Omega\cap(B_\rho(x_0)\cup B_\rho(x_1))} \left|\frac{1}{|x-y|}-\frac{1}{|\tau(x)-\tau(y)|}\right|\,dy\,dx\\
				&\le c\rho^2=o(\rho),
			\end{aligned}
		\end{equation}
		for some universal $c>0$.
		By \eqref{eq:h0}, \eqref{eq:h1} and \eqref{eq:h2} we deduce \eqref{eq:claimh}, and the proof is concluded.
	\end{proof}
	
	%

	We are now in position to show $C^{2,\alpha}-$regularity of   the boundary of a minimizer $\Om$. This can be done in two steps: first one shows that such a boundary is locally the graph of a $C^{2,\alpha}$ function defined on the boundary of a ball. To do that one exploits the improvement of flatness technique from~\cite[Section~7 and~8]{alca},  readapted with minimal changes to our setting as in~\cite[Appendix]{gush}. Then, as we already know by the previous section that the boundary of $\Om$ is close in Hausdorff distance to that of a ball, we obtain that the local parametrization is a global parametrization of class $C^{2,\alpha}$ on the boundary of the ball.
	{We first need a definition (see~\cite[Definition~7.1]{alca}).
		\begin{definition}
			Let $\gamma_\pm\in(0,1]$ and $k>0$. A weak solution $u$ of~\eqref{eq:bdp4.31} is of class $F(\gamma_-,\gamma_+,k)$ in $B_\rho(x_0)$ with respect to direction $\nu\in \mathbb{S}^{N-1}$ if 
			\begin{enumerate}
				\item[(a)]$x_0\in \partial \{u>0\}$ and \[
				\begin{split}
					u=0,\qquad &\text{for }(x-x_0)\cdot \nu\leq -\gamma_-\rho,\quad x\in B_\rho(x_0),\\
					u(x)\geq \mu_{u}(x_0)[(x-x_0)\cdot \nu-\gamma_+\rho],\qquad &\text{for }(x-x_0)\cdot \nu\geq \gamma_+\rho,\quad x\in B_\rho(x_0).
				\end{split}
				\]
				\item[(b)] $|\nabla u(x_0)|\leq \mu_u(x_0)(1+k)$ in $B_\rho(x_0)$ and ${\rm osc}_{B_{\rho}(x_0)}\mu_u\leq k\mu_u(x_0)$.
			\end{enumerate} 
		\end{definition}
		We note that when $k=+\infty$, then condition $(b)$ is automatically satisfied.}
	We can show the following result.
	\begin{theorem}\label{thm:bdvthm4.18}
          Let $q\in (0,q_3 ]$, $\Omega$ be an optimal set
          for~\eqref{eq:minmin1}, and $u$ a function attaining
          $E_q(\Omega)$ and a weak solution to~\eqref{eq:bdp4.31} in
          $B_R$.
		Then there
		are constants $\overline \gamma$ and $\overline k$, depending only on
		$R$, $ \mu_u$, such that
		if $u$ is of class $F(\gamma,1,+\infty)$ in $B_{4\rho}(x_0)$ with respect to some direction $\nu\in \mathbb{S}^{N-1}$ with $\gamma\leq \overline \gamma$ and $\rho\leq \overline k \gamma^2$,
		then there exists a $C^{2, \alpha}$ function
		$f\colon \R^{2}\to \R$ with
		$\|f\|_{C^{2,\alpha}}\leq C(R, \mu_u)$ such that,
		calling
		\[ {\rm graph}_\nu f:=\{x\in \R^3 : x\cdot \nu
		=f(x-(x\cdot\nu)\nu)\},
		\]
		then \[
		\partial \{u>0\}\cap B_\rho(x_0)=(x_0+{\rm graph}_\nu(f))\cap B_\rho(x_0).
		\]
		Moreover for all $\eps_0>0$ there exists $q_\eps\in
                (0,q_3 ]$ such that if $q<q_\eps$ then
		\[
		\partial \{u>0\}=\left\{\left(r+\varphi\left(\tfrac{x}{|x|}\right)\right)\frac{x}{|x|}\,:\, x\in\partial B_r  \right\}
		\]
		where $\varphi\colon \partial B_1\to \R$ is a function with the same regularity of $f$ and $\|\varphi\|_{C^{2,\alpha}}\leq \eps_0$.
	\end{theorem}
	We omit the proof which is identical to that in \cite[Theorems 1.2 and
	6.8]{maru} (which is in turn inspired by~\cite[Theorem~8.1]{alca} and~\cite[Theorem~2.17 and Appendix]{gush}). {We note that in our setting $\mu_u$ is constant, thus the requirement to be $C^{1,\alpha}$ regular is trivially satisfied.} 
	
	We are now in position to prove Theorem~\ref{thm:mainbdd}.
	\begin{proof}[Proof of Theorem~\ref{thm:mainbdd}]
		The existence of a minimizer follows from Lemma~\ref{le:existminG} and Theorem~\ref{thm:noconstraint}.
		On the other hand, the fact that any optimal set is $C^{2,\alpha}$ nearly spherical follows from Theorem~\ref{thm:bdvthm4.18}.
	\end{proof}

	\section{The surgery result and the proof of Theorem~\ref{thm:main}}\label{sec:unbounded}
	In this section, we {remove} the equiboundedness assumption that
	was present in Theorem~\ref{thm:mainbdd}.  The surgery strategy that we
	employ is very similar to the one proposed in~\cite{mp} (see
	also~\cite{buma}) and used for the spectral Gamow problem
	in~\cite{maru}.  We recall here, for the reader's sake, the main
	notations and the changes that are needed in our setting {and we give a proof of the following main result}.
	\begin{lemma}\label{le:surgery}
          There exist universal constants $D$, $\overline \delta<1$
          and $\overline q\in (0,q_3 ]$ such that if
          $\q\leq \overline \q$ then for any open and connected set
          $\Om\subset \R^3$ of measure $|B_1|$ satisfying
          $E_q(\Omega)-\lambda_0(B)\leq \overline \delta$ there
          exists an open, connected set $\widehat \Om$ of measure
          $|B_1|$ with diameter bounded by $D$ and such that
		\[
		E_q(\widehat \Om)\leq E_q(\Om).
		\]
	\end{lemma}

	Let us introduce some notation.  Let $\Omega$ be a connected set of
	measure $|B_1|$ such that
	$\lambda_0(\Om)-\lambda_0(B_1)\leq E_q(\Omega)-\lambda_0(B_1) \leq \overline
	\delta$, and {we fix $B_1$ the ball attaining the minimum in the Fraenkel asymmetry for $\Omega$ (see~\eqref{eq:fraenel}). We can clearly assume (up to a traslation of $\Omega$) that $B_1$ is centered at the origin. Then,} by the quantitative Faber--Krahn inequality (see
	Theorem~\ref{thm:quantitativefk} or~\cite{brdeve}), we have 
	\[ |\Om\Delta B_1|=\mathcal A(\Om)\leq |B_1|^
	{1/3}\left(\frac{\overline \delta}{\widehat \sigma}\right)^{1/2},
	\]
	where $\widehat \sigma$ is the constant from Theorem~\ref{thm:quantitativefk}.
	By defining 
		\begin{equation}\label{eq:defK}
			K:=\lambda_0(B_1)+1\geq \lambda_0(B_1)+\overline \delta
		\end{equation} 
		we obtain immediately 
		\[
		E_q(\Omega)\leq K,\qquad\text{and in particular,}\qquad \int_{\Omega}|\nabla u|^2\, dx\leq K,
		\] 
	where $u=u_{\q,\Omega}$ from now on is the function attaining $E_q(\Omega)$.
	We then note that (since $B_1$ has unit radius)\[
	|\Om\setminus [- t,t]^3|\leq |\Om\Delta B|= \mathcal A(\Om),\qquad \text{for all }t\geq 1.
	\]
	Let $\mh\in (0, 1/4)$ be such that
	\begin{equation}\label{defmh}
		\frac{(4\mh)^{\frac 23}}{\lambda_0(B_1)|B_1|^{\frac23}} \, K \leq \frac 12\,.
	\end{equation}
	Moreover, we choose $\overline \delta$ small enough so that 
	\begin{equation}\label{eq:asimmpiccolamhat}
		|\Om\setminus [-1 ,1 ]^3|\leq \mathcal A(\Om)\leq |B_1|\left(\frac{\overline \delta}{\widehat \sigma}\right)\leq\frac{\mh}{2^{6}}.
	\end{equation}
	We first focus on the direction $e_1$ and detail the construction in this case.
	We shall denote $z=(x,y)\in \R\times\R^{2}$ and by $z_i$ the $i$-th component of $z\in \R^3$.
	For any $t\in \R$, we  define
	\begin{equation*}
		\Om_t:=\Big\{y\in\R^{2} : (t,y)\in\Om\Big\}\,,
	\end{equation*}
	and given any set $\Omega\subseteq\R^3$, we define its $1$-dimensional projections for $ p\in\{1,2,3\}$ as
	\[
	\pi_p(\Omega) := \Big\{ t\in\R:\, \exists \, (z_1,\, z_2, z_3)\in\Omega,\, z_p = t\Big\}\,.
	\]
	For every $t\leq -1$ we call
	\begin{align}\label{int0}
		\Omega^+(t) := \Big\{(x,y)\in\Om : x>t\Big\}\,, && \Omega^-(t) := \Big\{(x,y)\in\Om : x<t\Big\}\,, && \eps(t):=\mathcal H^{2}(\Om_t)\,.
	\end{align}
	Observe that 
	\begin{equation}\label{int1}
		m(t) := \big| \Omega^-(t) \big| = \int_{-\infty}^t \eps(s)\,ds\leq 2\mh\,.
	\end{equation}
	We call $u$ { the optimizer for
		$E_q(\Omega)$ (we note that it is unique since $\Omega=\{u>0\}$ is connected). } We define then also, for every $t\leq -1$,
	\begin{align}\label{int2}
		\de(t):=\int_{\Om_t}{|\nabla u(t,y)|^2\,d\hc^{2}(y)}\,, && \mu(t):=\int_{\Om_t}{u(t,y)^2\,d\hc^{2}(y)}\,, 
	\end{align}
	which makes sense since $u$ is smooth inside $\Om$. 
	Applying the Faber--Krahn inequality in $\R^{2}$ to the set $\Omega_t$, and using the rescaling property of eigenvalues on $\R^{2}$, we know that
	\[
	\eps(t) \lambda_0(\Omega_t)=\H^{2}(\Omega_t) \lambda_0(\Omega_t) \geq \lambda_0(B_{\R^2})\,,
	\]
	{calling $B_{\R^2}$ the ball of unit measure in $\R^{2}$.} As a trivial consequence, we can estimate $\mu$ in
	terms of $\eps$ and $\de$: in fact, noting that
	$u(t,\cdot)\in H^{1}_0(\Omega_t)$ and writing
	$\nabla u = (\nabla_1 u, \nabla_y u)$, we have
	\begin{equation}\label{eq:muest}
		\mu(t)=\int_{\Om_t}{u(t,\cdot)^2\,d\hc^{2}}\leq \frac{1}{\lambda_0(\Om_t)}\int_{\Om_t}{|\nabla_y u(t,\cdot)|^2\,d\hc^{2}}\leq C\eps(t)\de(t).
	\end{equation}
	We can now present two estimates which assure that $u$ and $\nabla u$
	cannot be too big in $\Om^-(t)$.
	
	\begin{lemma}\label{primastima}
		{Let $\Omega\subseteq \R^3$ and $u$ be as in Lemma~\ref{le:surgery}.} For every $t\leq -1$ the following inequalities hold: 
		\begin{align}\label{eq:udu-}
			\int_{\Omega^-(t)} u^2 \, dx\leq C_1 \eps(t)^{\frac{1}{2}}\de(t)\,, &&
			\int_{\Omega^-(t)} |\nabla u|^2 \, dx\leq C_1 \eps(t)^{\frac{1}{2}}\de(t)\,,
		\end{align}
		for some universal constant $C_1 {>0}$ .
	\end{lemma}
		The proof of the above Lemma follows, up to a few minor changes, as in~\cite[Lemma~2.3]{mp}, by working on $u$ (and recalling it solves the PDE~\eqref{eq:ELhartree}) instead of the first eigenfunction of the Dirichlet Laplacian in $\Omega$. We reproduce it here for the sake of completeness.
		\begin{proof}
			Let us fix $t\leq -1$. Consider the set $\Om_S^-$ obtained by the union of $\Om^-(t)$ and its reflection with respect to the plane $\{x=t\}$, and call $u_S \in H^{1}_0(\Om_S)$ the function obtained by reflecting $u$. Using the Faber-Krahn inequality, we find then
			\[
			\frac{\lambda_0(B_1)|B_1|^{\frac23}}{\big(2m(t)\big)^{\frac{2}{3}}} = \frac{\lambda_0(B_1)|B_1|^{\frac 23}}{|\Om_S^-|^{\frac{2}{3}}}
			\leq \lambda_0(\Om_S^-) \leq \frac{\begin{aligned}\int_{\Om^-_S}|\nabla u_S|^2\,dx\end{aligned}}{\begin{aligned}\int_{\Om^-_S}u_S^2\,dx\end{aligned}}= \frac{\begin{aligned}\int_{\Om^-(t)}|\nabla u|^2\,dx\end{aligned}}{\begin{aligned}\int_{\Om^-(t)}u^2\,dx\end{aligned}}
			= \frac{\begin{aligned} \int_{\Om^-(t)} |\nabla u|^2\,dx\end{aligned}}{\begin{aligned} \int_{\Om^-(t)} u^2\,dx \end{aligned}}\,,
			\]
			by the symmetry of $\Om^-_S$, and using the scaling. This estimate gives
			\begin{equation}\label{eq:eq1}
				\int_{\Om^-(t)} u^2\,dx \leq \frac{\big(2m(t)\big)^{\frac 23}}{\lambda_0(B_1)|B_1|^{\frac23}}\, \int_{\Om^-(t)} |\nabla u|^2\,dx
			\end{equation}
			which in particular, being $m(t)\leq 2\mh$ and recalling~\eqref{defmh}, implies
			\begin{equation}\label{estray2}
				\int_{\Omega^-(t)} u^2\,dx \leq \frac 12\,.
			\end{equation}
			On the other hand, recalling that $-\Delta u \leq  \lambda_q u$ in $\Omega$, by Schwarz inequality and using~\eqref{eq:muest} we have
			\begin{equation}\label{eq:eq2}\begin{split}
					\int_{\Om^-(t)} |\nabla u|^2\,dx &\leq   \int_{\Om^-(t)} \lambda_q u^2\,dx + \int_{\Om_t}  u \,\frac{\partial u}{\partial \nu}\,d\mathcal H^2
					\leq K \int_{\Om^-(t)} u^2\,dx + \sqrt{\int_{\Om_t} u^2\,d\mathcal H^2  \int_{\Om_t} |\nabla u|^2\,d\mathcal H^2}\\
					&\leq K \int_{\Om^-(t)} u^2\,dx + C \eps(t)^{\frac 1{2}} \delta(t)\,.
			\end{split}\end{equation}
			It is now easy to obtain~\eqref{eq:udu-} combining~\eqref{eq:eq1} and~\eqref{eq:eq2}. In fact, by inserting the latter into the first, we find
			\[
			\int_{\Om^-(t)} u^2\,dx \leq  \frac{\big(2m(t)\big)^{\frac 23}}{\lambda_0(B_1)|B_1|^{\frac23}}\, \bigg( K \int_{\Om^-(t)} u^2\,dx + C \eps(t)^{\frac 1{2}} \delta(t) \bigg)\,,
			\]
			which by~(\ref{defmh}) again yields
			\begin{equation}\label{eq:leftudu-}
				\frac{1}{2} \int_{\Om^-(t)} u^2\,dx \leq  \frac{\big(2m(t)\big)^{\frac 23}}{\lambda_0(B_1)|B_1|^{\frac23}}\, C \eps(t)^{\frac 1{2}} \delta(t) 
				\leq C \eps(t)^{\frac 1{2}} \delta(t) \,.
			\end{equation}
			The left estimate in~\eqref{eq:udu-} is then obtained. To get the right one, one has then just to insert~\eqref{eq:leftudu-} into~\eqref{eq:eq2}.
		\end{proof}
	
	Let us go further into the construction, giving some additional definitions. For any $t\leq -1$ and $\sigma(t)>0$, we define the cylinder $Q(t)$ as
	\begin{equation}\label{defcyl}
		Q(t):=\Big\{(x,y)\in \R^3: \, t-\sigma(t) < x < t,\ (t,y) \in\Om\Big\} = \big(t-\sigma(t),t\big) \times \Omega_t\,,
	\end{equation}
	where for any $t\leq -1$ we set
	\begin{equation}\label{defsigma}
		\sigma(t)= \eps(t)^{\frac 1{2}}\,.
	\end{equation}
	We let also $\Omt(t)=\Omega^+(t)\cup Q(t)$, and we introduce $\ut\in H^1_0\big(\Omt(t)\big)$ as
	\begin{equation}\label{utilde}
		\ut(x,y):=\left\{
		\begin{array}{ll}
			u(x,y) &\hbox{if $(x,y)\in \Omega^+(t)$}\,, \\[5pt]
			\begin{aligned}
				\frac{x-t+\sigma(t)}{\sigma(t)}\,u(t,y)
			\end{aligned}&\hbox{if $(x,y)\in Q(t)$}\,.
		\end{array}
		\right.
	\end{equation}
	The fact that $\ut$ vanishes on $\partial\Omt(t)$ is obvious; moreover, $\nabla u=\nabla \ut$ on $\Omega^+(t)$, while on $Q(t)$ one has
	\begin{equation}\label{estdut}
		\nabla \ut(x,y)=\left(\frac{u(t,y)}{\sigma(t)}\,,\,\frac{x-t+\sigma(t)}{\sigma(t)}\,\nabla_y u(t,y)\right)\,.
	\end{equation}
	
	A simple calculation allows us to estimate the integrals of $\ut$ and $\nabla \ut$ on $Q(t)$.
	\begin{lemma}\label{lemmatest}
		For every $t\leq -1$, one has
		\begin{align}\label{newtest}
			\int_{Q(t)}|\nabla \ut|^2 \, dx\leq C_2\eps(t)^{\frac{1}{2}}\de(t) \,, &&
			\int_{Q(t)} \ut^2 \, dx\leq C_2 \eps(t)^{\frac 3{2}} \delta(t)\,,
		\end{align}
		for a universal constant $C_2>0$.
	\end{lemma}
	The proof of the above Lemma follows as~\cite[Lemma~2.4]{mp}.

	Another simple but useful estimate concerns the Rayleigh quotients of the functions $\ut$ on the sets $\Omt(t)$: notice that, while $u$ has unit $L^2$ norm, the modifed function $\ut$ in general is not normalized so we need to take care also of its norm.
	
	\begin{lemma}\label{noth}
		There exists a universal constant $C_3>0$ such that for every $t\leq -1$, one has
		\begin{equation}\label{est1a}
			\int_{\Omt(t)}|\nabla \ut|^2\, dx\leq \int_\Omega|\nabla u|^2\, dx+ C_3\eps(t)^{\frac 1 {2}}\de(t)\,,\qquad \int_{\Omt(t)}\ut^2\, dx\geq \int_\Omega u^2\, dx-C_3\eps(t)^{\frac 1 {2}}\de(t)\,.
		\end{equation}
	\end{lemma}
	\begin{proof}
		It is enough to note that, by definition of $\Omt(t)$ and using Lemma~\ref{primastima} and~\ref{lemmatest}, we obtain for the gradient term\[
		\begin{split}
			\int_{\Omt(t)}|\nabla \ut|^2\, dx&=\int_{\Om^+(t)}|\nabla u|^2\, dx+\int_{Q(t)}|\nabla \ut|^2\, dx\\
			&=\int_\Omega|\nabla u|^2\, dx+\int_{Q(t)}|\nabla \ut|^2\, dx-\int_{\Om^-(t)}|\nabla u|^2\, dx\leq \int_\Omega|\nabla u|^2\, dx+C_2\eps(t)^{\frac 1 {2}}\de(t)\,,
		\end{split}
		\]
		while for the function, we have \[
		\int_{\Omt(t)} \ut^2\, dx=\int_{\Om^+(t)} u^2\, dx+\int_{Q(t)} \ut^2\, dx=\int_\Omega u^2\, dx+\int_{Q(t)}\ut^2\, dx-\int_{\Om^-(t)} u^2\, dx\geq \int_\Omega u^2\, dx-C_1\eps(t)^{\frac 1 {2}}\de(t)\,.
		\]
	\end{proof}

	We can now enter in the central part of our construction. Basically, we aim to show that either $\Omega$ already has bounded left ``tail'' in direction $e_1$, or some rescaling of $\Omt(t)$ has energy lower than that of $\Omega$. 
	\begin{lemma}\label{threeconditions}
          Let $\Omega$ be as in the assumptions of
          Lemma~\ref{primastima}, and let $t\leq -1$. There exist
          universal $\overline \q\in(0,q_3]$ and $C_4>2$ such
          that, for all $\q\leq \overline \q$ exactly one of the three
          following conditions hold:
		\begin{enumerate}
			\item[(1) ] $\max\big\{ \eps(t),\, \delta(t) \big\} > 1$;
			\item[(2) ] (1) does not hold and $m(t) \leq C_4 \big( \eps(t) + \delta(t)\big) \eps(t)^{\frac 1{2}}$;
			\item[(3) ] (1) and~(2) do not hold and one has that\[
			\frac{\int_{\Omh(t)}|\nabla \uh|^2\, dx}{\int_{\Omh(t)}\uh^2\, dx}\leq \int_\Omega|\nabla u|^2\, dx,\qquad \text{and}\qquad
			E\big(\Omh(t)\big)< E_q(\Omega),
			\] where for $t\leq -1$  we set
			\[
			\Omh(t) := \big|B_1\big|^{\frac13}\big| \Omt(t) \big|^{-\frac 13} \Omt(t),\qquad \text{and} \qquad \uh(x)=\ut\big(|B_1|^{-\frac13}|\Omt(t)|^{\frac13} x \big),\quad\text{for $x\in \Omh(t)$.}
			\]
		\end{enumerate}
	\end{lemma}
	\begin{proof}
		Assume~(1) is false. Then it is possible to apply Lemma~\ref{noth}, to obtain
		\begin{equation}\label{putinto}
			\int_{\Omt(t)}|\nabla \ut|^2\, dx\leq \int_\Omega|\nabla u|^2\, dx+ C_3\eps(t)^{\frac 1 {2}}\de(t)\,,\qquad \int_{\Omt(t)}\ut^2\, dx\geq \int_\Omega u^2\, dx-C_3\eps(t)^{\frac 1 {2}}\de(t)=1-C_3\eps(t)^{\frac 1 {2}}\de(t)\,.
		\end{equation}
		By the scaling properties of the eigenvalue and the fact that $\big| \Omh(t)\big|=|B_1|$, we know that
		\[
		\frac{\int_{\Omh(t)}|\nabla \uh|^2\, dx}{\int_{\Omh(t)}\uh^2\, dx}= \frac{\big| \Omt(t) \big|^{\frac 23}}{|B_1|^{\frac23}} 
		\frac{\int_{\Omt(t)}|\nabla \ut|^2\, dx}{\int_{\Omt(t)}\ut^2\, dx}\,.
		\]
		By construction,
		\[
		\big| \Omt(t)\big|= \big| \Om^+(t)\big|+ \big|Q(t)\big| = |B_1| - m(t) + \eps(t)^{\frac 3{2}}\,,
		\]
		hence the above estimates, the scaling of the integrals due to the definition of $\uh$ and~\eqref{putinto} lead to
		\begin{equation}\label{muchhere}\begin{split}
				\frac{\int_{\Omh(t)}|\nabla \uh|^2\, dx}{\int_{\Omh(t)}\uh^2\, dx} &= \Big( 1 - \frac{m(t)}{|B_1|} + \frac{\eps(t)^{\frac32}}{|B_1|} \Big)^{\frac 23}\, \frac{\int_{\Omt(t)}|\nabla \ut|^2\, dx}{\int_{\Omt(t)}\ut^2\, dx},  \\
				&\leq\Big( 1 - \frac{2}{3|B_1|} \, m(t) +  \frac{2}{3|B_1|}\, \eps(t)^{\frac 3{2}} \Big) \Big(1+C_3\eps^{\frac{1}{2}}(t)\delta(t)\Big)  \Big( \int_\Omega|\nabla u|^2\, dx+C_3\eps(t)^{\frac 1 {2}}\delta(t) \Big)\\
				&\leq \left(\int_\Omega|\nabla u|^2\, dx - \frac{2\lambda_0(B_1)}{3|B_1|}\, m(t)+ \frac{2K}{3|B_1|}\, \eps(t)^{\frac 3{2}}+\bigg(2C_3+KC_3+ \frac{2}{3|B_1|}\bigg)\eps(t)^{\frac 1 {2}}\delta(t)\right)\,.
			\end{split}
		\end{equation}

		At this point, defining $C_4:= \max{\{\frac{2(K+1)}{3|B_1|}+2C_3+KC_3,2\}}$, if \[
		m(t) \leq C_4 \big( \eps(t) + \delta(t)\big) \eps(t)^{\frac 1{2}},
		\] 
		then condition~(2) holds true. 
		Otherwise, we immediately have that 
		\begin{equation}\label{eq:stimala1}
			\frac{\int_{\Omh(t)}|\nabla \uh|^2\, dx}{\int_{\Omh(t)}\uh^2\, dx} \leq \left(\int_\Omega|\nabla u|^2\,dx-\left(\frac{2\lambda_0(B_1)}{3|B_1|}-1\right)m(t)\right)\leq \int_\Omega|\nabla u|^2\, dx-C_5m(t),
		\end{equation}
		for a universal constant $C_5>0$, therefore the first part of the third claim is verified.
		
		On the other hand, we note that, using the $L^\infty$ bound of $u$, see Lemma~\ref{lem:inftybound}, the fact that $\ut\leq u$ by construction and also~\cite[Lemma~2.4]{fp}, \[
		D(\ut^2,\ut^2)=D(u^2,u^2)+2\int_{\Om^+(t)}\int_{Q(t)}\frac{\ut^2(x)\ut^2(y)}{|x-y|}\,dxdy+\int_{Q(t)}\int_{Q(t)}\frac{\ut^2(x)\ut^2(y)}{|x-y|}\,dxdy\leq D(u^2,u^2)+ C_{fp}\eps^{\frac{3}{2}}(t).
		\]
		Then we can estimate, using the appropriate scalings,  
		\begin{equation}\label{eq:stimaV}
			\begin{split}
				&\frac{D(\uh^2,\uh^2)}{\left(\int_{\Omh(t)}\uh^2\, dx\right)^2}\leq \frac{D(\ut^2,\ut^2)}{\left(\int_{\Omt(t)}\ut^2\, dx\right)^2}\left(1-\frac{m(t)}{|B_1|}+\frac{\eps(t)^{\frac{3}{2}}}{|B_1|}\right)^{-\frac{2}{3}}\leq \Big(1+\frac{2}{3|B_1|}m(t)\Big)\frac{D(\ut^2,\ut^2)}{\left(\int_{\Omt(t)}\ut^2\, dx\right)^2}\\
				&\leq  \Big(1+\frac{2}{3|B_1|}m(t)\Big)\Big(1+C_3\eps^{\frac{1}{2}}(t)\delta(t)\Big)\Big(D(u^2,u^2)+C_{fp}\eps^{\frac{3}{2}}(t)\Big)\\
				&\leq D(u^2,u^2) +C\|u\|^2_{L^\infty}m(t)+C_{fp}\eps^{\frac{3}{2}}(t)+C_3\|u\|^2_{L^\infty}\eps^{\frac{1}{2}}(t)\delta(t)\\
				&\leq D(u^2,u^2) +C\|u\|^2_{L^\infty}m(t)+(C_{fp}+C_3\|u\|^2_{L^\infty})m(t)\\
				&=D(u^2,u^2)+C_6 m(t).
			\end{split}
		\end{equation}
		Then, putting together~\eqref{eq:stimala1} and~\eqref{eq:stimaV}, recalling also Remark~\ref{rmk:scaleinvariant} for the equivalence of the scale invariant energy, 
		\begin{equation}
			\begin{split}
				& E_q(\Omh(t))\leq  \frac{\int_{\Omh(t)}|\nabla \uh|^2\, dx}{\int_{\Omh(t)}\uh^2\, dx} +\frac{\ q}{2} \frac{D(\uh^2,\uh^2)}{\left(\int_{\Omh(t)}\uh^2\, dx\right)^2}\\
				&\leq \int_\Omega|\nabla u|^2\, dx+\frac{q}{2} D(u^2,u^2)-(C_5-\frac{\q}{2} C_6)m(t)\\
				&\leq \int_\Omega|\nabla u|^2\, dx+\frac{\ q}{2} D(u^2,u^2)-\frac{C_5}{2}m(t),
			\end{split}
		\end{equation}
		up to taking $\q\leq \overline \q<\sqrt{\frac{C_5}{2C_6}}$, so that in
		this case condition $(3)$ holds and the proof is concluded.
	\end{proof}
	
	Once we have Lemma~\ref{threeconditions}, the rest of the proof follows as in~\cite{mp} or~\cite{maru} as we detail here below.
	
	\begin{proof}[Proof of Lemma~\ref{le:surgery}]
		It is enough to repeat the {analogs} of~\cite[Lemma~8.7,
		Lemma~8.8, Proposition~8.1 and Section~9.2]{maru}, noting that it is
		only a geometric argument and having $\int_\Omega|\nabla u|^2\, dx$
		instead of $\lambda_0(\Omega)$ does not change anything.
	\end{proof}

	\begin{proof}[Proof of Theorem~\ref{thm:main}]
		We aim to apply the surgery result Lemma~\ref{le:surgery} and then to employ Theorem~\ref{thm:mainbdd}. Precisely, first, as in Section $9.2$ of \cite{maru} we select a minimizing sequence for problem~\eqref{eq:minmin1} made of connected sets. Then, by Lemma \ref{le:surgery} we select another minimizing sequence of equibounded sets.  At this point we are in position to apply Theorem~\ref{thm:mainbdd} and we conclude.	
	\end{proof}

  \section{{The case of $q$ large}: proof of
    Theorem~\ref{thm:nomin}}\label{sec:nonexistence}
  In this final section, we show that the energy of a minimizer cannot
  exceed a value of order $q^{3/2}$. This is done by a simple estimate
  on the energy of a suitably chosen union of balls with mutual
  distance large enough. As a consequence, we show that for large
  values of $q$ any minimizer has a bound on the diameter {from
    below} and that the ball $B_1$ can not be optimal.  We also
  formulate the following conjecture, motivated by the proof of
  Lu-Otto~\cite{LuOtto} for the Thomas-Fermi-{Dirac-}Von Weizs\"acker
  energy.
		\begin{conjecture}
			There exists a threshold $M>0$ such that for $q>M$ no minimizer occurs for \eqref{eq:minmin1}.
		\end{conjecture}

\begin{proof}[Proof of Theorem \ref{thm:nomin}]
  We construct a competitor $\Omega_N$ made up of a suitably chosen
  quantity of disjoint balls with mutual distance diverging to
  infinity.
	
  Let $\Omega_N=\cup_{i=1}^NB_r(x_i)$, where $|x_i-x_j|$ is diverging
  {sufficiently fast} to infinity for $i\not=j$ {as
    $q \to \infty$}.  We select $N\in \N$ and $r>0$ so that
  $|\Omega_N|=|B_1|$; this implies in particular that {$Nr^3 = 1$}.
  Calling $w_B\in H^1_0(B_1)$ the first Dirichlet eigenfunction of
  $B_1$ {extended by zero to all of $\R^3$ and} normalized with
  $\int_{B_1}w_B^2(z)\,dz=1$, then we can define as test function
  {supported} on $\Omega_N$ the function
  {$\widetilde w_N = \sum_{i=1}^N  w_B((x - x_i) /
    r)$ }, so that
  \[ \int_{\Omega_N}\widetilde
    w_N^2(z)\,dz=N\int_{B_r}w_B^2(z/r)\,dz=N r^3
    \int_{B_1}w_B^2(y)dy=\int_{B_1}w_B^2=1.
\]
Thus, using the minimality of $\Omega_N$ {we obtain} \[
\begin{split}
  &E_q(\Omega_N)\leq E_q(\widetilde w_N,\Omega_N) {\leq}
  N\int_{B_r}|\nabla w_B(z/r)|^2\,dz+\frac{{N} q}{2}\int_{B_r}
  \int_{B_r}\frac{w_B^{2}(z/r)w_B^{2}(w/r)}{|z-w|}\,dzdw
  {+ {C q \over \min_{1 \leq i < j \leq N} |x_i - x_j|} } \\
  &\le C\left(N^{2/3}\int_{B_1}|\nabla
    w_B|^2\,dy+\frac{q}{{N^{2/3}}}D(w_B^2,w_B^2)\right){=}
  C\left( \lambda_0(B_1)N^{2/3}+\frac{q}{{N^{2/3}}} {D(w_B^2,
      w_B^2)} \right).
\end{split}
\]
Minimizing with respect to $N$ {for $q$ sufficiently large
  universal}, we obtain that the optimal number of balls to be
$N=C q^{3/4}$ for {some $C = C(q)>0$ approaching a universal
  constant as $q \to \infty$}, leading
to\[ E_q(\Omega_N)\leq C q^{1/2}
\]
{for all $q$ sufficiently large.}  As a consequence, if $\Omega$ is
an optimal set for problem~\eqref{eq:minmin1}, then we have the bound
from above
\begin{equation}\label{eq:boundm32}
E_q(\Omega)\leq C q^{1/2}
\end{equation}
{for all $q$ sufficiently large.}

On the other hand, we can estimate from below the energy of the ball
of unit radius, for $q\geq 1$:
\[ E_q (B_1)= {\min_{u \in H^1_0(B_1) } \left\{ \int_\Omega|\nabla
      u|^2 dx +\frac{q}2 D( u^2, u^2) : \int_{B_1} u^2 dx = 1
    \right\} \geq \frac{q}{4}}.
\]
As a consequence, for ${q}$ sufficiently large the unit ball cannot
be the optimal set.

Finally, let $\Omega$ be an optimal set for
problem~\eqref{eq:minmin1}. Then
\[ E_q(\Omega)\geq
  \frac{q}2\int_\Omega\int_\Omega\frac{u^2(x)u^2(y)}{|x-y|}\,dxdy\geq
  \frac{q}2\int_\Omega\int_\Omega
  \frac{u^2(x)u^2(y)}{\mathrm{diam}(\Omega)}\,dxdy\geq
  \frac12\frac{q}{\mathrm{diam}(\Omega)}.
\]
Thus, thanks to~\eqref{eq:boundm32}, we deduce that \[
\mathrm{diam}(\Omega)\geq Cq^{1/2},
\]
for {all $q$ sufficiently large and a universal constant $C>0$}.
\end{proof}

	\appendix

	\section{{The physical model and}
		non-dimensionalization} \label{appendix}
	
	In its dimensional form, the ground state bosonic Hartree energy for
	the Cooper pairs takes the form (in the SI units)
	\begin{equation}
		\label{eq:Hartreedim}
		\mathcal E( u, \Omega) =  \frac{N \hbar^2}{2 m^*} \int_\Om|\nabla
		 u(x)|^2 \, d 
		x + \frac{N (N - 1) e^2}{2 \pi \eps_0 \eps} \int_\Om \int_\Om
		\frac{ u^2(x)  u^2(y)}{|x - y|} \, d x \, d y,
	\end{equation}
	where $m^*$ is the effective mass of a Cooper pair and $-2|e|$ is
	its charge, where $-|e|$ is the elementary charge, $\eps_0$ is the
	vacuum permittivity, $\eps$ is the dielectric constant of the
	surrounding matrix (within a simplified local treatment of the
	dielectric), $N$ is the number of Cooper pairs in the island and
	$ u$ is a single-orbital wave function subject to the
	normalization
	\begin{equation}
		\label{eq:phinormdim}
		\int_\Om  u^2(x) \, d x = 1.
	\end{equation}
	
	We now perform a rescaling
	\begin{equation}
		\label{eq:rescaling}
		x \to L x, \qquad  u \to L^{-3/2}  u,
	\end{equation}
	that keeps the normalization condition in \eqref{eq:phinormdim}
	unchanged. After some simple algebra we arrive at
	\begin{equation}
		\label{eq:EErescaled}
		\mathcal E(L^{-3/2} u(\cdot /L), L \Omega) = \frac{N \hbar^2}{2 m^* L^2}
		E_q(u, \Omega), 
	\end{equation}
	where
	\begin{equation}
		\label{eq:qsquared}
		q = \frac{2 e^2 (N - 1) m^* L}{\pi \hbar^2 \eps_0 \eps}. 
	\end{equation}
	With the choice of $L = \left( \tfrac{3 V}{4 \pi} \right)^{1/3}=( \tfrac{V}{|B_1|})^{1/3}$ we
	then arrive at the shape optimization problem in \eqref{eq:minmin1}.


\begin{thebibliography}{10}

\bibitem{aac}
N.~Aguilera, H.~W. Alt, and L.~A. Caffarelli.
\newblock An optimization problem with volume constraint.
\newblock {\em SIAM J. Control Optim.}, 24(2):191--198, 1986.

\bibitem{alca}
H.~W. Alt and L.~A. Caffarelli.
\newblock Existence and regularity for a minimum problem with free boundary.
\newblock {\em J. Reine Angew. Math.}, 325:105--144, 1981.

\bibitem{amti}
L.~Ambrosio and P.~Tilli.
\newblock {\em Topics on analysis in metric spaces}, volume~25 of {\em Oxford
  Lecture Series in Mathematics and its Applications}.
\newblock Oxford University Press, Oxford, 2004.

\bibitem{beloborodov}
I.~Beloborodov, A.~Lopatin, V.~Vinokur, and K.~Efetov.
\newblock Granular electronic systems.
\newblock {\em Reviews of Modern Physics}, 79(2):469--518, 2007.

\bibitem{BenguriaBrezisLieb}
R.~Benguria, H.~Br\'{e}zis, and E.~H. Lieb.
\newblock The {T}homas-{F}ermi-von {W}eizs\"{a}cker theory of atoms and
  molecules.
\newblock {\em Comm. Math. Phys.}, 79(2):167--180, 1981.

\bibitem{bouchiat}
V.~Bouchiat, D.~Vion, P.~Joyez, D.~Esteve, and M.~Devoret.
\newblock Quantum coherence with a single cooper pair.
\newblock volume~76, pages 165--170, 1998.

\bibitem{brdeve}
L.~Brasco, G.~De~Philippis, and B.~Velichkov.
\newblock Faber-{K}rahn inequalities in sharp quantitative form.
\newblock {\em Duke Math. J.}, 164(9):1777--1831, 2015.

\bibitem{brafra2012}
L.~Brasco and G.~Franzina.
\newblock A note on positive eigenfunctions and hidden convexity.
\newblock {\em Arch. Math. (Basel)}, 99(4):367--374, 2012.

\bibitem{BrianconHayouniPierre}
T.~Brian\c{c}on, M.~Hayouni, and M.~Pierre.
\newblock Lipschitz continuity of state functions in some optimal shaping.
\newblock {\em Calc. Var. Partial Differential Equations}, 23(1):13--32, 2005.

\bibitem{bu}
D.~Bucur.
\newblock Minimization of the {$k$}-th eigenvalue of the {D}irichlet
  {L}aplacian.
\newblock {\em Arch. Ration. Mech. Anal.}, 206(3):1073--1083, 2012.

\bibitem{buma}
D.~Bucur and D.~Mazzoleni.
\newblock A surgery result for the spectrum of the {D}irichlet {L}aplacian.
\newblock {\em SIAM J. Math. Anal.}, 47(6):4451--4466, 2015.

\bibitem{bmpv}
D.~Bucur, D.~Mazzoleni, A.~Pratelli, and B.~Velichkov.
\newblock Lipschitz regularity of the eigenfunctions on optimal domains.
\newblock {\em Arch. Ration. Mech. Anal.}, 216(1):117--151, 2015.

\bibitem{buttiker}
M.~B\"uttiker.
\newblock Zero--current persistent potential drop across small-capacitance
  josephson junctions.
\newblock {\em Physical Review B}, 36(7):3548--3555, 1987.

\bibitem{cmt17}
R.~Choksi, C.~B. Muratov, and I.~Topaloglu.
\newblock An old problem resurfaces nonlocally: {G}amow's liquid drops inspire
  today's research and applications.
\newblock {\em Notices Amer. Math. Soc.}, 64(11):1275--1283, 2017.

{\bibitem{davies}
E.~B. Davies.
\newblock Heat kernels and spectral theory.
\newblock Cambridge Tracts in Mathematics, 92, Cambridge Univ. Press, Cambridge, 1989.
}

\bibitem{demamu}
G.~De~Philippis, M.~Marini, and E.~Mukoseeva.
\newblock The sharp quantitative isocapacitary inequality.
\newblock {\em Rev. Mat. Iberoam.}, 37(6):2191--2228, 2021.

\bibitem{fp}
N.~Fusco and A.~Pratelli.
\newblock Sharp stability for the {R}iesz potential.
\newblock {\em ESAIM Control Optim. Calc. Var.}, 26:Paper No. 113, 24, 2020.

\bibitem{GT}
D.~Gilbarg and N.~S. Trudinger.
\newblock {\em Elliptic partial differential equations of second order}, volume
  224 of {\em Grundlehren der mathematischen Wissenschaften [Fundamental
  Principles of Mathematical Sciences]}.
\newblock Springer-Verlag, Berlin, second edition, 1983.

\bibitem{gnr1}
M.~Goldman, M.~Novaga, and B.~Ruffini.
\newblock Existence and stability for a non-local isoperimetric model of
  charged liquid drops.
\newblock {\em Arch. Ration. Mech. Anal.}, 217(1):1--36, 2015.

\bibitem{gnr2}
M.~Goldman, M.~Novaga, and B.~Ruffini.
\newblock On minimizers of an isoperimetric problem with long-range
  interactions under a convexity constraint.
\newblock {\em Anal. PDE}, 11(5):1113--1142, 2018.

\bibitem{gnr4}
M.~Goldman, M.~Novaga, and B.~Ruffini.
\newblock Rigidity of the ball for an isoperimetric problem with strong
  capacitary repulsion.
\newblock {\em JEMS}, In press. DOI 10.4171/JEMS/1451, 2024.

\bibitem{gush}
B.~Gustafsson and H.~Shahgholian.
\newblock Existence and geometric properties of solutions of a free boundary
  problem in potential theory.
\newblock {\em J. Reine Angew. Math.}, 473:137--179, 1996.

\bibitem{hartree}
D.~Hartree.
\newblock The wave mechanics of an atom with a non-{C}oulomb central field.
  part {I}: {T}heory and methods.
\newblock {\em Mathematical Proceedings of the Cambridge Philosophical
  Society}, 24(1):89 -- 110, 1928.

\bibitem{Heinonen}
J.~Heinonen, T.~Kilpel\"{a}inen, and O.~Martio.
\newblock {\em Nonlinear potential theory of degenerate elliptic equations}.
\newblock Dover Publications, Inc., Mineola, NY, 2006.
\newblock Unabridged republication of the 1993 original.

\bibitem{henrot}
A.~Henrot.
\newblock {\em Extremum problems for eigenvalues of elliptic operators}.
\newblock Frontiers in Mathematics. Birkh\"{a}user Verlag, Basel, 2006.

\bibitem{HenrotPierre}
A.~Henrot and M.~Pierre.
\newblock {\em Shape variation and optimization.  A geometrical analysis}, volume~28 of {\em EMS Tracts
  in Mathematics}.
\newblock European Mathematical Society (EMS), Z\"{u}rich, 2018.

\bibitem{kjaergaard}
M.~Kjaergaard, M.~E. Schwartz, J.~Braum\"uller, P.~Krantz, J.~I.-J. Wang,
  S.~Gustavsson, and W.~D. Oliver.
\newblock Superconducting qubits: Current state of play.
\newblock {\em Annual Review of Condensed Matter Physics}, 11:369 -- 395, 2020.

\bibitem{km1}
H.~Kn\"{u}pfer and C.~B. Muratov.
\newblock On an isoperimetric problem with a competing nonlocal term {I}: {T}he
  planar case.
\newblock {\em Comm. Pure Appl. Math.}, 66(7):1129--1162, 2013.

\bibitem{km}
H.~Kn\"{u}pfer and C.~B. Muratov.
\newblock On an isoperimetric problem with a competing nonlocal term {II}:
  {T}he general case.
\newblock {\em Comm. Pure Appl. Math.}, 67(12):1974--1994, 2014.

\bibitem{lebris05}
C.~Le~Bris and P.-L. Lions.
\newblock From atoms to crystals: a mathematical journey.
\newblock {\em Bull. Amer. Math. Soc. (N.S.)}, 42(3):291--363, 2005.

\bibitem{LiebLoss}
E.~H. Lieb and M.~Loss.
\newblock {\em Analysis}, volume~14 of {\em Graduate Studies in Mathematics}.
\newblock American Mathematical Society, Providence, RI, second edition, 2001.

\bibitem{lions}
P.-L. Lions.
\newblock Solutions of {H}artree-{F}ock equations for {C}oulomb systems.
\newblock {\em Comm. Math. Phys.}, 109(1):33--97, 1987.

\bibitem{LuOtto}
J.~Lu and F.~Otto.
\newblock Nonexistence of a minimizer for {T}homas-{F}ermi-{D}irac-von
  {W}eizs\"{a}cker model.
\newblock {\em Comm. Pure Appl. Math.}, 67(10):1605--1617, 2014.

\bibitem{mazya}
V.~G. Maz'ja.
\newblock {\em Sobolev spaces}.
\newblock Springer Series in Soviet Mathematics. Springer-Verlag, Berlin, 1985.
\newblock Translated from the Russian by T. O. Shaposhnikova.

\bibitem{mp}
D.~Mazzoleni and A.~Pratelli.
\newblock Existence of minimizers for spectral problems.
\newblock {\em J. Math. Pures Appl. (9)}, 100(3):433--453, 2013.

\bibitem{maru}
D.~Mazzoleni and B.~Ruffini.
\newblock A spectral shape optimization problem with a nonlocal competing term.
\newblock {\em Calc. Var. Partial Differential Equations}, 60(3):Paper No. 114,
  46, 2021.

\bibitem{MurNovRuf}
C.~B. Muratov, M.~Novaga, and B.~Ruffini.
\newblock On equilibrium shape of charged flat drops.
\newblock {\em Comm. Pure Appl. Math.}, 71(6):1049--1073, 2018.

\bibitem{MNR2022}
C.~B. Muratov, M.~Novaga, and B.~Ruffini.
\newblock Conducting flat drops in a confining potential.
\newblock {\em Arch. Ration. Mech. Anal.}, 243(3):1773--1810, 2022.

\bibitem{velectures}
B.~Velichkov.
\newblock {\em Regularity of the one-phase free boundaries}, volume~28 of {\em
  Lecture Notes of the Unione Matematica Italiana}.
\newblock Springer, 2023.

\bibitem{vion}
D.~Vion.
\newblock Josephson quantum bits based on a {C}ooper pair box.
\newblock {\em Les Houches Summer School Proceedings}, 79(C):521--523, 2004.


\end{thebibliography}

\end{document}